\documentclass[10 pt]{amsart}

\usepackage{amsmath,amssymb,amsthm,amscd}
\usepackage{graphicx}
\usepackage[margin=1.5in]{geometry}
\usepackage{tikz}
\usepackage[OT2,T1]{fontenc}
\usepackage{comment}
\usepackage{tabulary}

\usepackage{amsfonts,amstext,verbatim}
\usepackage{enumerate,color,stmaryrd,psfrag} 
\usepackage{url}
\usepackage{mdwlist}   % lists and itemize 
\usepackage{mathtools} % special math symbols (eg: xarrows)
\usepackage{wasysym}   % some special caracters (hexagon)

\usepackage{dsfont}

\newtheorem{thm}{Theorem}[section]
\newtheorem{lem}[thm]{Lemma}
\newtheorem{pro}[thm]{Proposition}
\newtheorem{cor}[thm]{Corollary}

\newtheorem{conj}[thm]{Conjecture}

\theoremstyle{remark} 
\newtheorem{rem}[thm]{Remark}

\theoremstyle{definition}

\newcommand{\GL}{\mathrm{GL}}
\newcommand{\Gal}{\mathrm{Gal}}
\newcommand{\blegendre}[2]{\left (\frac{#1}{#2}\right)}

\newcommand{\Frob}[0]{\mathrm{Frob}}

\newcommand{\ord}[0]{\mathrm{ord}}

\DeclareSymbolFont{cyrletters}{OT2}{wncyr}{m}{n}
\DeclareMathSymbol{\Sha}{\mathalpha}{cyrletters}{"58}

\begin{document}

\pagestyle{plain}
\setcounter{page}{1}

\title{Character sums for elliptic curve densities}
\author{Julio Brau}

\begin{abstract}
\noindent If $E$ is an elliptic curve over $\mathbb{Q}$, then it follows from work of Serre and Hooley that, under the assumption of the Generalized Riemann Hypothesis, the density of primes $p$ such that the group of $\mathbb{F}_p$-rational points of the reduced curve $\tilde{E}(\mathbb{F}_p)$ is cyclic can be written as an infinite product $\prod \delta_\ell$ of local factors $\delta_\ell$ reflecting the degree of the $\ell$-torsion fields, multiplied by a factor that corrects for the entanglements between the various torsion fields. We show that this correction factor can be interpreted as a character sum, and the resulting description allows us to easily determine non-vanishing criteria for it. We apply this method in a variety of other settings. Among these, we consider the aforementioned problem with the additional condition that the primes $p$ lie in a given arithmetic progression. We also study the conjectural constants appearing in Koblitz's conjecture, a conjecture which relates to the density of primes $p$ for which the cardinality of the group of $\mathbb{F}_p$-points of $E$ is prime. \end{abstract}

\maketitle

\section{Introduction}
The motivation for this paper comes from the classical conjecture of Artin from 1927 which predicts the density of primes $p$ for which a given rational number is a primitive root modulo $p$. More precisely, let $g$ be an integer different from $\pm 1$, and let $h$ be the largest integer such that $g=g_0^h$ with $g_0\in \mathbb{Z}$. The heuristic reasoning described by Artin was the following. If $p$ is a prime number coprime to $g$, then $g$ is a primitive root modulo $p$ if and only if there is no prime $\ell$ dividing $p-1$ such that $g\equiv y^\ell \pmod{p}$ for some $y$. Note that this congruence condition can be given as a splitting condition on the prime $p$ in the field $F_\ell:=\mathbb{Q}(\zeta_\ell, \sqrt[\ell]{g})$. Indeed, the condition on $p$ is equivalent to $p$ not splitting completely in the aforementioned field. In other words, $g$ is a primitive root modulo $p$ if and only if for every prime $\ell<p$ we have that $\Frob_p$ is not the identity element in $\Gal(F_\ell/\mathbb{Q})$.

For a fixed $\ell$, the density of primes which do not split completely in $F_\ell$ is equal to 
$$
\delta_\ell := 1-\frac{1}{[F_\ell : \mathbb{Q}]},
$$
and this equals $1-\frac{1}{\ell-1}$ for $\ell \mid h$ and $1-\frac{1}{\ell(\ell-1)}$ otherwise. If we assume the splitting conditions in the various fields $F_\ell$ to be independent, then it is reasonable to expect that the density of primes $p$ for which $g$ is a primitive root modulo $p$ is equal to $\prod_\ell \delta_\ell$. This was the density originally conjectured by Artin, however years later (see \cite{StevenhagenCF}) he himself noticed that this assumption of independence is not correct, as the fields $F_\ell$ can have non-trivial intersections. If $F_2 = \mathbb{Q}(\sqrt{g})$ has discriminant $D\equiv 1\pmod{4}$, then $F_2$ is contained in the compositum of the fields $F_\ell$ with $\ell \mid D$. The corrected version of the conjecture was proven by Hooley under the assumption of the Generalized Riemann Hypothesis (GRH). He showed in \cite{Hooley} that, conditional on GRH, the density of primes such that $g$ is a primitive root modulo $p$ equals
\begin{equation}\label{HooleySum}
C_g = \sum_{n=1}^\infty \frac{\mu(n)}{[F_n:\mathbb{Q}]}
\end{equation}
where $F_n = \mathbb{Q}(\zeta_n,\sqrt[n]{g})$ and $\mu$ is the M\"{o}bius function. In the same paper Hooley shows that (\ref{HooleySum}) can be rewritten as
\begin{equation}\label{HooleyProduct}
C_g = \mathfrak{C}_g \prod_{\ell \mid h} \Big( 1- \frac{1}{\ell-1} \Big) \prod_{\ell \nmid h} \Big( 1- \frac{1}{\ell(\ell-1)} \Big),
\end{equation}
where $\mathfrak{C}_g$ is an \emph{entanglement correction factor}, a rational number which depends on $g$. In fact it is given explicitly by
$$
\mathfrak{C}_g:=1-\prod_{\substack{\ell \mid D \\ \ell \mid h}} \frac{-1}{\ell-2} \cdot \prod_{\substack{\ell \mid D \\ \ell \nmid h}} \frac{-1}{\ell^2-\ell-1}.
$$
One advantage of having $C_g$ in the form given by (\ref{HooleyProduct}) is that it makes it easy to see when the density $C_g$ vanishes. Vanishing of $C_g$ implies that, conjecturally, there exist only finitely many primes $p$ such that $g$ is a primitive root modulo $p$, and the multiplicative structure of $C_g$ and $\mathfrak{C}_g$ allows one to identify precisely when and why this can happen.

There are many interesting generalisations to Artin's conjecture on primitive roots. For instance, one could consider only primes $p$ which lie in a prescribed congruence class modulo some integer $f$. One could also study the set of primes $p$ such that $g$ generates a subgroup of a given index in $(\mathbb{Z}/p \mathbb{Z})^\times$. As is shown in \cite{Lenstra}, in both of these cases one can again obtain a density under GRH via a formula similar to (\ref{HooleySum}). However, it is not clear how to describe the non-vanishing criteria of such densities from such a sum. 

In \cite{LMP}, the authors develop an efficient method to compute entanglement correction factors $\mathfrak{C}_g$ for Artin's original conjecture and several of its generalisations. Their method consists in expressing $\mathfrak{C}_g$ as a sum of quadratic characters. More precisely, they show that $\mathfrak{C}_g$ has the form
$$
\mathfrak{C}_g = 1 + \prod_{\ell} E_\ell
$$
where each $E_\ell$ is the average value of a character $\chi_\ell$ over an explicit set. One crucial fact used to arrive at this form is that when $D\equiv 1 \pmod{4}$, then for $n$ divisible by $2D$ we have that the subgroup
$$
\Gal(F_n/\mathbb{Q}) \hookrightarrow \prod_{\ell \mid n} \Gal(F_\ell /\mathbb{Q})
$$
is cut out by a quadratic character $\chi$ measuring the nature of the intersections of the fields $F_\ell$. The structure of $C_g$ as an Euler product and the description of $\mathfrak{C}_g$ naturally lead to non-vanishing criteria. 

In this paper we attempt to generalize this method to the setting of elliptic curves. There are many problems concerning the study of the set of primes $p$ such that the reduced curve $\tilde{E}(\mathbb{F}_p)$ satisfies a certain condition. One of these arises as a natural analogue of Artin's conjecture on primitive roots. Namely, given an elliptic curve $E$ over $\mathbb{Q}$, the problem is to determine the density of primes $p$ such that $\tilde{E}(\mathbb{F}_p)$ is cyclic. The first thing to note is that the condition of $\tilde{E}(\mathbb{F}_p)$ being cyclic is completely determined by the splitting behaviour of $p$ in the various torsion fields $\mathbb{Q}(E[\ell])$ for different $\ell$. Given this, we can proceed similarly by defining local densities $\delta_\ell$ and attempting to find the entanglement correction factor $\mathfrak{C}_E$, however one quickly runs into various difficulties which were not present in the case of classical Artin. One of these is that it is not necessarily true that $\Gal(\mathbb{Q}(E[m])/\mathbb{Q}) \hookrightarrow \prod_{\ell\mid m} \Gal(\mathbb{Q}(E[\ell])/\mathbb{Q})$ is a normal subgroup and even if so, the quotient need not be $\{\pm 1\}$ or even abelian for that matter.

This leads us to the study in Section \ref{AbEntang} of \emph{abelian entanglements}. If $G$ is a subgroup of $G_1\times\dots\times G_n$ such that the projection maps $\pi_i:G\rightarrow G_i$ are surjective for $1\leqslant i \leqslant n$, then we show that $G$ is normal in $G_1\times \dots\times G_n$ with abelian quotient if and only if $G$ has abelian entanglements.

In Section \ref{CurvesAbEntang} we define elliptic curves with abelian entanglements to be those elliptic curves with the property that $G(m_E)$ has abelian entanglements in the sense of Section \ref{AbEntang}. We show that this definition is equivalent to $\mathbb{Q}(E[m_1]) \cap \mathbb{Q}(E[m_2])$ being an abelian extension of $\mathbb{Q}$ for every coprime $m_1,m_2$. It is for this class of curves that we will be able to apply our character sum method, with Theorem \ref{ThmGoodFrobeniusCalc} being a crucial ingredient.

Section \ref{CyclicRedEC} applies Theorem \ref{ThmGoodFrobeniusCalc} to the aforementioned problem of cyclic reduction of elliptic curves. We explicitly evaluate the density $C_E$ as an Euler product $\prod_\ell \delta_\ell $ times an entanglement correction factor $\mathfrak{C}_E$. We then compute $\mathfrak{C}_E$ in the case of Serre curves and give examples of a few other elliptic curves with more complicated Galois Theory, as well as establishing non-vanishing criteria for these conjectural densities.

In Section \ref{CyclicRedAP} we study a variant of the problem of cyclic reduction on elliptic curves. Namely, we impose the additional condition that $p$ lie in a prescribed congruence class modulo some integer $f$. This introduces new difficulties as the splitting conditions on $p$ become more complicated, but it also illustrates the way in which our method can be used to handle a variety of different scenarios. In the end the computation of $\mathfrak{C}_E$ is again reduced to fairly mechanical local computations. Again Serre curves and several other examples are treated in detail. 

Section \ref{KobConj} we study a different type of problem. We look at a classical conjecture of Koblitz on the asymptotic behaviour of the number of primes $p$ for which the cardinality of $\tilde{E}(\mathbb{F}_p)$ is prime. We see that the character sum approach can also be applied to describe the constant appearing in this asymptotic. In this case there are not even conditional results, and the constant computed is purely conjectural. However the constant we compute has previously been described via different methods by Zywina in \cite{ZywKob}, where he provides some convincing numerical evidence for it.

The study of conjectural constants led us to investigate the class of elliptic curves with abelian entanglements, and naturally leads to the question of whether there exist elliptic curves whose entanglements are not all abelian. To be precise, can one classify the triples $(E,m_1,m_2)$ with $E$ an elliptic curve over $\mathbb{Q}$ and $m_1,m_2$ a pair of coprime integers for which the entanglement field $\mathbb{Q}(E[m_1]) \cap \mathbb{Q}(E[m_2])$ is non-abelian over $\mathbb{Q}$?  In Section \ref{NonAbEntang} we exhibit an infinite family of elliptic curves for which this is the case.

\vspace{5 mm}

\noindent \textbf{Acknowledgements.} I would like to thank Peter Stevenhagen for the discussion that originally led to this work, as well as Hendrik Lenstra and Nathan Jones for many helpful conversations. While writing this paper I was partially supported by EPSRC grant EP/M016838/1 as well as Leiden University.

\section{Abelian entanglements}\label{AbEntang}
In this section we define the property of having abelian entanglements and study some of its consequences. We will first give some useful preliminaries on fibered products of groups.

\subsection{Fibered products of groups}

Let $G_1$, $G_2$  and $Q$ be groups, $\psi_1:G_1\rightarrow Q$, $\psi_2:G_2\rightarrow Q$ be surjective homomorphisms, and let $\psi$ denote the abbreviation for the ordered pair $(\psi_1,\psi_2)$. We define the \emph{fibered product} of  $G_1$ and $G_2$ over $\psi$, denoted $G_1 \times_\psi G_2$, to be the group
\begin{equation}\label{FibProd}
G_1 \times_\psi G_2 := \{(g_1,g_2) \in G_1 \times G_2 : \psi_1(g_1) = \psi_2(g_2) \}
\end{equation}
Note that $G_1 \times_\psi G_2$ is a subdirect product of $G_1$ and $G_2$, that is, it is a subgroup of $G_1\times G_2$ which maps surjectively onto $G_1$ and $G_2$ under the canonical projection homomorphisms. The following is a well-known lemma which tells us that the converse of this also holds. We present a proof here for completeness and because it will be useful later on in this section.

\begin{lem}[Goursat's Lemma]\label{Goursat}
Let $G_1$ and $G_2$ be groups and let $G\subseteq G_1\times G_2$ be a subgroup such that the projections $\pi_1:G\rightarrow G_1$ and $\pi_2:G\rightarrow G_2$ are surjective. Then there exists a group $Q$ and surjective homomorphisms $\psi_1:G_1\rightarrow Q$, $\psi_2:G_2\rightarrow Q$ such that $G=G_1\times_\psi G_2$. That is,
$$
G = \{(g_1,g_2) \in G_1 \times G_2 : \psi_1(g_1) = \psi_2(g_2) \}.
$$
\end{lem}

\begin{proof}
 Let $N_1 = (G_1 \times \{1\})\cap G$ and $N_2 = (\{1\} \times G_2)\cap G$, where we use $1$ to denote the identity elements of both $G_1$ and $G_2$. Then $N_1 = \ker \pi_2$ and $N_2 = \ker \pi_1$. We now show that $\pi_1(N_1) \unlhd G_1$ and $\pi_2(N_2) \unlhd G_2$. Note that $N_1 \unlhd G$ as it is the kernel of $\pi_2$, and let $g_1\in G_1$. Then as $\pi_1:G\rightarrow G_1$ is surjective, there exists $g_2\in G_2$ such that $g:=(g_1,g_2) \in G$. It follows that
$$
g_1 \pi_1(N_1) = \pi_1(g)\pi_1(N_1) = \pi_1(gN_1) = \pi_1(N_1g) = \pi_1(N_1)\pi_1(g) = \pi_1(N_1)g_1
$$
so $\pi_1(N_1) \unlhd G_1$ as claimed. 	Similarly we have $\pi_2(N_2) \unlhd G_2$. Note that $\pi_i(N_i) \simeq N_i$ and hence $(G_i\times \{1\}) / N_i \simeq G_i/N_i$. Consider the map $f:G\rightarrow G_1/N_1 \times G_2/N_2$ defined by $(g_1,g_2) \mapsto (g_1N_1, g_2N_2)$ where we have written $N_i$ in place of $\pi_i(N_i)$. One can easily check that for $(g_1,g_2)\in G$ one has
$$
g_1N_1 = N_1 \Longleftrightarrow g_2N_2 = N_2
$$
hence the image of $f$ is the graph of a well-defined isomorphism $G_1/N_1 \xrightarrow{\sim} G_2/N_2$. The result now follows from setting $Q:=G_2/N_2$.
\end{proof}
We will refer to the $N_i$ in the proof as \emph{Goursat subgroups} and to $Q$ as the \emph{Goursat quotient} associated to this fibered product. 

Suppose now that $L_1/K, L_2/K$ are Galois extensions of fields, with $G_i = \Gal(L_i/K)$ and $G = \Gal(L_1L_2/K)$, where $L_1L_2$ denotes the compositum of $L_1$ and $L_2$. Then it is well known from Galois theory that 
$$
G = \{(g_1,g_2) \in G_1\times G_2 : g_1 \mid_{L_1\cap L_2} = g_2 \mid_{L_1\cap L_2}\} \leqslant G_1\times G_2.
$$

\begin{lem}\label{GoursatGalois}
Keeping the above notation, we have that
$$
G = G_1 \times_{\psi} G_2 
$$
with $\psi_i: G_i \rightarrow \Gal(L_1\cap L_2 / K)$ the canonical restriction maps.
\end{lem}

\begin{proof}
From the proof of Goursat's lemma, $N_1 = (G_1 \times \{1\})\cap G$ and $\pi_1(N_1)$ is the subgroup of $G_1$ which acts trivially on $L_1\cap L_2$, and the result follows.
\end{proof}

\subsection{Groups with abelian entanglements}
Suppose $G$ is a subgroup of $G_1\times\dots\times G_n$ such that the projection maps $\pi_i:G\rightarrow G_i$ are surjective for $1\leqslant i \leqslant n$. We will concern ourselves here with the situation when $G$ is normal in $G_1\times \dots\times G_n$ with abelian quotient.

For a group $G$, we will denote by $G'$ the commutator subgroup of $G$, and for $x,y\in G$, $[x,y]=x^{-1}y^{-1}xy$ will denote the commutator of $x$ and $y$.  For a non-empty subset $S\subset\{1,\dots, n\}$ we write $\pi_S$ for the projection map
$$
\pi_S:G_1\times\dots\times G_n \longrightarrow \prod_{i\in S} G_i
$$
and let $G_S$ denote the image of $G$ under this projection map. Note that for each partition $\sqcup_j T_j = \{1,\dots , n\}$ we have a canonical inclusion
$$
G\xhookrightarrow{\phantom{hello}} \prod_j G_{T_j}.
$$

Let $\mathcal{P}:=\{S,T\}$ be a partition of $\{1,\dots, n\}$, so that $S\sqcup T = \{1,\dots , n\}$. Then $G$ is a subdirect product of $G_S \times G_T$ so by Goursat's lemma there is a group $Q_{\mathcal{P}}$ and a pair of homomorphisms $\psi_{\mathcal{P}} := (\psi_{\mathcal{P}}^{(1)},\psi_{\mathcal{P}}^{(2)})$ with
\begin{align*}
\psi_{\mathcal{P}}^{(1)} \colon &G_S \longrightarrow Q_{\mathcal{P}} \\
\psi_{\mathcal{P}}^{(2)} \colon &G_T \longrightarrow Q_{\mathcal{P}} 
\end{align*}
such that $G = G_S \times_{\psi_{\mathcal{P}}} G_T$. We say that $G$ has \emph{abelian entanglements with respect to $G_1\times\dots\times G_n$} if $Q_{\mathcal{P}}$ is abelian for each two-set partition $\mathcal{P}$ of $\{1,\dots ,n\}$. We will often write only that $G$ has abelian entanglements, omitting with respect to which direct product of groups if this is clear from the context.

\begin{pro}\label{entanglementsProp}
Let $G$ be a subgroup of $G_1\times\dots\times G_n$ such that the projection maps $\pi_i:G\rightarrow G_i$ are surjective for $1\leqslant i \leqslant n$. Keeping the notation as above, $G$ is a normal subgroup of $G_1 \times \dots \times G_n$ if and only if $G$ has abelian entanglements.
\end{pro}

The proof will use the following proposition, which is the case $n=2$.

\begin{pro}\label{entanglementsN2}
Let $G$ be a subgroup of $G_1 \times G_2$ such that the projection maps $\pi_1:G\rightarrow G_1$ and $\pi_2:G\rightarrow G_2$ are surjective. Then $G\unlhd G_1 \times G_2$ if and only if $G$ has abelian entanglements.
\end{pro}

\begin{proof}
Suppose first that $G$ has abelian entanglements, and let $x:=(x_1,x_2)\in G$. We will show that for any $a\in G_1\times\{1\}$ one has $axa^{-1}\in G$, and similarly for every $b\in \{1\}\times G_2$. The result will then follow. So take $a:=(a_1, 1)\in G_1\times\{1\}$.  Let $N_1$ and $N_2$ be the corresponding Goursat subgroups associated to $G$, that is, $N_1 = (G_1 \times \{1\})\cap G$ and $N_2 = (\{1\} \times G_2)\cap G$. Then because $G$ has abelian entanglements we have that $(G_1 \times \{1\})/N_1$ is abelian, or equivalently $(G_1 \times \{1\})' \leqslant N_1$. It follows that $[(a_1,1),(x_1,1)]\in G$, however
\begin{align*}
[(a_1,1),(x_1,1)] &= (a_1,1)(x_1,1)(a_1,1)^{-1} (x_1,1)^{-1} \\
						  &= (a_1,1)(x_1,x_2)(a_1,1)^{-1} (x_1,x_2)^{-1}
\end{align*}
and $(x_1,x_2)^{-1}$ is in $G$, hence $(a_1,1)(x_1,x_2)(a_1,1)^{-1}$ is also in $G$, as claimed. Similarly one can show $(1,b_2)(x_1,x_2)(1,b_2)^{-1} \in G$ for any $b_2 \in G_2$, and we conclude $G$ is normal in $G_1 \times G_2$.

For the converse, suppose that $G \unlhd G_1 \times G_2$. We will show that $(G_1 \times \{1\})' \leqslant N_1$, from which it follows that $G$ has abelian entanglements. Let $(x_1,1)$ and $(y_1,1)$ be arbitrary elements of $G_1 \times \{1\}$. Because $\pi_1:G\rightarrow G_1$ is surjective, there exists $z\in G_2$ such that $(y_1,z)\in G$. As $G \unlhd G_1\times G_2$, we have $(x_1,1)(y_1,z)(x_1,1)^{-1}$ is in $G$ and hence so is $[(x_1,1),(y_1,z)]$. Using the fact that $[(x_1,1),(y_1,1)] = [(x_1,1),(y_1,z)]$, we obtain $[(x_1,1),(y_1,1)] \in G$. However $[(x_1,1),(y_1,1)] = ([x_1,y_1],1) \in G_1 \times \{1\}$, hence the result.
\end{proof}

\begin{proof}[Proof of Proposition \ref{entanglementsProp}]
Again we suppose first that $G$ has abelian entanglements, and we proceed similarly as in the case $n=2$. Let $x:=(x_1,\dots, x_n)\in G$, and for $j\in\{1,\dots,n\}$ let $a:=(1,\dots,1,a_j,1,\dots,1) \in \{1\} \times \dots \times \{1\} \times G_j \times \{1\} \times \dots \times \{1\}$ where the $a_j$ is in the $j$-th position. Let $S_j := \{1,\dots ,n\} \backslash \{j\}$. Then $G\leqslant G_j \times G_{S_j}$ with surjective projection maps and the corresponding quotient $(G_j \times \{1\})/N_j$ is abelian. By Proposition \ref{entanglementsN2}, $G$ is a normal subgroup of $G_j \times G_{S_j}$. But $a$ is certainly an element of $G_j \times G_{S_j}$, hence $axa^{-1}\in G$. Since $j$ was chosen arbitrarily we conclude $G\unlhd G_1\times\dots\times G_n$.

Conversely, suppose $G\unlhd G_1\times\dots\times G_n$, and let $\mathcal{P}:=\{S,T\}$ be a partition of $\{1,\dots ,n\}$. Then note that $G_S\times G_T$ may be viewed as a subgroup of $G_1\times \dots \times G_n$ and so $G\unlhd G_S \times G_T$. By Proposition \ref{entanglementsN2} the corresponding Goursat quotient $Q_{\mathcal{P}}$ is abelian, hence $G$ has abelian entanglements. This completes the proof.
\end{proof}

In the proof we used the subset $S_j := \{1,\dots n\} \backslash \{j\} \subset \{1,\dots, n\}$. Here we have that $G$ is a subdirect product of $G_j \times G_{S_j}$, so by Goursat's lemma there is a group $Q_j$ and a pair of homomorphisms $\psi_j := (\psi_j^{(1)},\psi_j^{(2)})$ such that $G = G_j \times_{\psi_j} G_{S_j}$. The following corollary tells us that these are all the partitions we need to consider in order to determine whether or not $G$ has abelian entanglements.

\begin{cor}\label{entanglementsCor}
With the notation above, $G$ has abelian entanglements if and only if $Q_j$ is abelian for every $j\in \{1,\dots , n\}$.
\end{cor}

\begin{proof}
One implication is trivial. Suppose that $Q_j$ is abelian for every $j\in \{1,\dots ,n\}$. Then by the proof of Proposition \ref{entanglementsProp}, $G$ is a normal subgroup of $G_1\times\dots\times G_n$, and again using Proposition \ref{entanglementsProp}, $G$ has abelian entanglements, as claimed.
\end{proof}

\begin{pro}\label{NormalQuotient}
Suppose that $G$ is a normal subgroup of $G_1\times \dots\times G_n$ such that the projection maps $\pi_i : G \rightarrow G_i$ are surjective for all $i$. Then the quotient $(G_1\times \dots \times G_n) / G$ is abelian.
\end{pro}

\begin{proof}
We will proceed by showing that $(G_1\times\dots\times G_n)'\leqslant G$. Let $x:=(x_1,\dots\,x_n) \in (G_1\times\dots\times G_n)'$. By Proposition \ref{entanglementsProp} $G$ has abelian entanglements, so for each $j$, to the inclusion $G\hookrightarrow G_j \times G_{S_j}$ there corresponds an abelian quotient $G_j/\pi_j(N_j)$, where $N_j = (G_j \times \{1\})\cap G$. The composition
$$
G_1\times\dots\times G_n \xrightarrow{\pi_j} G_j \longrightarrow G_j/\pi_j(N_j)
$$
gives an abelian quotient of $G_1\times\dots\times G_n$, hence $x_j = \pi_j(x_1,\dots,x_n)$ is contained in $\pi_j(N_j)$. It follows that $(1,\dots,1,x_j,1\dots,1)\in G$. As $j$ was arbitrary, and $\prod_j (1,\dots,1,x_j,1\dots,1) = x$, we conclude $x\in G$.
\end{proof}

\begin{pro}\label{projEntang}
Suppose $G$ has abelian entanglements with respect to $G_1\times\dots\times G_n$ and let $S\subseteq\{1,\dots,n\}$. Then $G_S$ has abelian entanglements with respect to $\prod_{i\in S} G_i$. 
\end{pro}

\begin{proof}
We will show that $G_S$ is normal in $\prod_{i\in S} G_i$. Note that 
$$
G \leqslant \pi_S^{-1} (G_S) \leqslant G_1\times\dots\times G_n
$$
and by Proposition \ref{NormalQuotient} the quotient $(G_1\times\dots\times G_n)/G$ is abelian. It follows then that $\pi_S^{-1} (G_S)$ is normal in $G_1\times\dots\times G_n$, and denote the quotient by $\Phi_S$. Now $\ker \pi_S \subset \pi_S^{-1} (G_S)$ so the map $G_1\times\dots\times G_n \rightarrow \Phi_S$ factors via $\prod_{i\in S} G_i$. Let $\psi_S$ be such that the following diagram commutes

\begin{center}
\begin{tikzpicture}[auto]
\node (a) at (0,2) {$G_1\times\dots\times G_n$};
\node (b) at (0,0) {$\displaystyle{\prod_{i\in S} G_i}$};
\node (c) at (2.2,0) {$\Phi_S$};

\node at (-.3,1.2) {$\pi_S$};
\node at (1.2,.3) {$\psi_S$};

\draw[->,thick] (a) -- (b);
\draw[->,thick] (a) -- (c);
\draw[->,thick] (b) -- (c);
\end{tikzpicture}.
\end{center}
It is easy to see that the kernel of $\psi_S$ is precisely $G_S$, hence $G_S$ is normal in $\prod_{i\in S} G_i$ and by Proposition \ref{entanglementsProp} $G_S$ has abelian entanglements with respect to $\prod_{i\in S} G_i$, as claimed.
\end{proof}

\section{Elliptic curves with abelian entanglements}\label{CurvesAbEntang}
We consider here a family of elliptic curves with the property that the intersections of the different torsion fields of each curve in this family are abelian extensions. 

We say that an elliptic curve $E$ has \emph{abelian entanglements} if the corresponding group $G(m_E) \leqslant G(\ell_1^{\alpha_1}) \times \dots \times G(\ell_n^{\alpha_n})$ has abelian entanglements in the sense of section \ref{AbEntang}, where $m_E$ as usual denotes the smallest split and stable integer for $E$, and has prime factorisation $m_E=\ell_1^{\alpha_1}\dots \ell_n^{\alpha_n}$.

\begin{lem}\label{EquivEntang}
The following two conditions are equivalent:
\begin{itemize}
\item[(i)]
$E$ has abelian entanglements.
\item[(ii)] 
For each $m_1,m_2\in \mathbb{N}$ which are relatively prime, the intersection
$$
\mathbb{Q}([m_1]) \cap \mathbb{Q}([m_2])
$$
is an abelian extension of $\mathbb{Q}$.
\end{itemize}
\end{lem}

\begin{proof}
Suppose $E$ has abelian entanglements, and let $m_1,m_2$ be relatively prime. If $m_1$ and $m_2$ both divide $m_E$, then by Proposition \ref{projEntang} gives that $G(m_1m_2)$ has abelian entanglements with respect to $G(m_1)\times G(m_2)$. This implies the Goursat quotient $Q_{m_1m_2}$ is abelian, and by Lemma \ref{GoursatGalois} we have $\mathbb{Q}([m_1]) \cap \mathbb{Q}([m_2])$ is an abelian extension of $\mathbb{Q}$. For general $m_1,m_2$, let
$$
m_1' = (m_1,m_E), \quad m_2' = (m_2,m_E).
$$
Then $m_1'$ and $m_2'$ are relatively prime integers dividing $m_E$ so be the same argument $\mathbb{Q}([m_1']) \cap \mathbb{Q}([m_2'])$ is an abelian extension of $\mathbb{Q}$. From Serre's open image Theorem if $n$ is any integer and $d$ is coprime to $nm_E$ then 
$$
G(nd) = G(n) \times \GL_2(\mathbb{Z}/d\mathbb{Z}).
$$
It follows that $Q_{m_1m_2}$ is isomorphic  to $Q_{m_1'm_2'}$, hence the claim.
\end{proof}

\begin{cor}\label{CorEquivEntang}
If $E$ has abelian entanglements, then for any $m:=\prod_i q_i^{s_i}$ we have that $G(m)\leqslant \prod_i G(q_i^{s_i})$ has abelian entanglements.
\end{cor}

\begin{proof}
This follows immediately from Corollary \ref{entanglementsCor} and Lemma \ref{EquivEntang}.
\end{proof}

Assume now that $E$ is an elliptic curve over $\mathbb{Q}$ with abelian entanglements, and let $m$ be a positive integer with prime factorisation $m=\prod_\ell \ell^{\alpha_\ell}$. Since $E$ has abelian entanglements, by Corollary  \ref{CorEquivEntang} and Proposition \ref{NormalQuotient} there are a map $\psi_m$ and a finite abelian group $\Phi_m$ that fit into the exact sequence
\begin{equation}\label{ExSeqPsi}
1 \longrightarrow G(m) \longrightarrow \prod_{\ell \mid m} G(\ell^{\alpha_\ell}) \xrightarrow{\phantom{a}\psi_m\phantom{b}} \Phi_m \longrightarrow 1 .
\end{equation}
Note that the group $\Phi_m$ measures the extent to which there are entanglements between the various $\ell^{\alpha_\ell}$-torsion fields. For instance $\Phi_m$ is trivial if and only if for any two coprime integers $m_1,m_2$ dividing $m$ one has $\mathbb{Q}(E[m_1])\cap \mathbb{Q}(E[m_2]) = \mathbb{Q}$. The following lemma tells us that $\Phi_{m_E}$ measures the full extent to which the distinct torsion fields of $E$ have any entanglements.

\begin{lem}\label{squareFreeM}
Let $m$ be a positive integer and $d$ be a positive integer coprime to $m_E$. Then $\Phi_{md} \simeq \Phi_{m}$.
\end{lem}

\begin{proof}
Again there is a map $\psi_{md}$ and an abelian group $\Phi_{md}$ which fit into the short exact sequence
$$
1 \longrightarrow G(md) \longrightarrow \prod_{\ell^{\alpha_\ell} || md} G(\ell^{\alpha_\ell}) \xrightarrow{\phantom{a}\psi_{md}\phantom{b}} \Phi_{md} \longrightarrow 1.
$$
As $d$ is coprime to $m_E$, by Serre's open image Theorem we have that
\begin{equation}\label{ConstantChi}
G(md) = G(m) \times \prod_{\ell^{\alpha_\ell} || d} G(\ell^{\alpha_\ell})
\end{equation}
It follows that $G(\ell^{\alpha_\ell})$ is contained in the kernel of $\psi_{md}$ for any $\ell \mid d$, hence $\Phi_{md}\simeq \Phi_m$.
\end{proof}

For each prime $\ell \mid m$, let $S(\ell)$ be a subset of $G(\ell^{\alpha_\ell})$, and define
$$
\mathcal{S}_m:= \prod_{\ell \mid m} S(\ell), \quad \mathcal{G}_m:= \prod_{\ell \mid m} G(\ell^{\alpha_\ell}).
$$
so that $\mathcal{S}_m \subset \mathcal{G}_m$. The following theorem allows us to compute the fraction of elements in $G(m)$ that belong to $\prod_{\ell \mid m} S(\ell)$. It will play a key role in the method we will develop for computing entanglement correction factors as character sums. If $A$ is an abelian group, then $\widehat{A}$ denotes the group of characters $\chi:A\rightarrow \mathbb{C}^\times$.

\begin{thm}\label{ThmGoodFrobeniusCalc}
Assume $E/\mathbb{Q}$ has abelian entanglements, and let $\Phi_m$ be as in (\ref{ExSeqPsi}). For each $\tilde{\chi}\in \widehat{\Phi}_m$ a character of $\Phi_m$, let $\chi$ be the character of $\mathcal{G}_m$ obtained by composing $\tilde{\chi}$ with $\psi_m$, and let $\chi_\ell$ the restriction of $\chi$ to the component $G(\ell^{\alpha_\ell})$. Then
$$
\frac{|\mathcal{S}_m\cap G(m)|}{|G(m)|} = \bigg( 1+\sum_{\tilde{\chi}\in \widehat{\Phi}_m-\{1\}}\prod_{\ell | m} E_{\chi,\ell} \bigg ) \frac{|\mathcal{S}_m|}{|\mathcal{G}_m|},
$$
where
$$
E_{\chi,\ell} = \sum_{x\in S(\ell)} \frac{\chi_\ell(x)}{|S(\ell)|}.
$$
\end{thm}

\begin{proof}
Let $\mathds{1}_{\mathcal{S}_m}$ be the indicator function of $\mathcal{S}_m$ in $\mathcal{G}_m$, and $\mathds{1}_{G(m)}$ that of $G(m)$. Also, to simplify notation we will use $\Phi$ in place of $\Phi_m$. Then we have that
$$
\frac{|\mathcal{S}_m\cap G(m)|}{|G(m)|} = \frac{1}{|G(m)|} \sum_{x\in \mathcal{G}_m} \mathds{1}_{\mathcal{S}_m} (x) \mathds{1}_{G(m)} (x).
$$
By the orthogonality relations of characters (see for instance \S VI.1 of \cite{SerreACA}) we have that if $x\in \mathcal{G}_m$, then
$$
\sum_{\tilde{\chi}\in \widehat{\Phi}} \chi(x) = 
\begin{cases}
[\mathcal{G}_m:G(m)] &\ \text{if $x \in G(m)$} \\
0	&\ \text{if $x\notin G(m)$}.
\end{cases}
$$
This implies that
$$
\mathds{1}_{G(m)} = \frac{1}{[\mathcal{G}_m:G(m)]}\sum_{\tilde{\chi}\in \widehat{\Phi}} \chi,
$$
so it follows that
\begin{align*}
\frac{|\mathcal{S}_m\cap G(m)|}{|G(m)|} &= \frac{1}{|\mathcal{G}_m|} \bigg ( \sum_{x\in \mathcal{G}_m} \mathds{1}_{\mathcal{S}_m}(x) + \sum_{x\in \mathcal{G}_m} \sum_{\tilde{\chi} \in \widehat{\Phi}\backslash\{1\}} \mathds{1}_{\mathcal{S}_m}(x) \chi(x) \bigg ) \\
                               &= \frac{|\mathcal{S}_m|}{|\mathcal{G}_m|} \bigg ( 1 + \sum_{x\in \mathcal{G}_m} \sum_{\tilde{\chi} \in \widehat{\Phi}\backslash\{1\}} \frac{\mathds{1}_{\mathcal{S}_m}(x) \chi(x)}{|\mathcal{S}_m|} \bigg )  \\
                               &= \frac{|\mathcal{S}_m|}{|\mathcal{G}_m|} \Bigg ( 1 + \sum_{\tilde{\chi} \in \widehat{\Phi}\backslash\{1\}} \bigg( \prod_{\ell \mid m} \sum_{x\in G(\ell)} \frac{\mathds{1}_{S(\ell)}(x) \chi_\ell(x)}{|S(\ell)|} \bigg) \Bigg ) \\
                               &= \frac{|\mathcal{S}_m|}{|\mathcal{G}_m|} \Bigg ( 1 + \sum_{\tilde{\chi} \in \widehat{\Phi}\backslash\{1\}} \bigg( \prod_{\ell \mid m} \sum_{x\in S(\ell)} \frac{\chi_\ell(x)}{|S(\ell)|} \bigg) \Bigg )
\end{align*}
where the third equality follows from the fact that $\mathds{1}_{\mathcal{S}_m}$ and $\chi$ are products of functions $\mathds{1}_{S(\ell)}$ and $\chi_\ell$ defined on the components $G(\ell^{\alpha_\ell})$. The result now follows from letting $E_{\chi,\ell} $ be the average value of $\chi_\ell$ on $S(\ell)$, that is
$$
E_{\chi,\ell} = \sum_{x\in S(\ell)} \frac{\chi_\ell(x)}{|S(\ell)|}.
$$
\end{proof}

%///////////// COMMENT /////////////////////////

\begin{comment}
Let
$$
m:=\prod_{\ell \mid m_E} \ell \quad \text{and} \quad L:=\prod_\ell \mathbb{Q}(E[\ell])
$$
denote the square-free part of $m_E$ as well as the composite over all primes $\ell$ of the $\ell$-torsion of $E$, respectively. The next lemma shows that when looking at entanglements which occur at the square-free torsion level, it suffices to consider the square-free part of $m_E$.

\begin{lem}\label{Squarefree}
With the above notation, we have
$$
\Big[\prod_\ell G(\ell) : \Gal(L/\mathbb{Q})\Big] = \Big[\prod_{\ell \mid m} G(\ell) : G(m)\Big].
$$
\end{lem}

\begin{proof}
That this index is finite is immediate from Serre's theorem (REF). Note that, because $m$ is the square-free part of $m_E$, we have
$$
\Gal(L/\mathbb{Q}) = G(m) \times \prod_{\ell \nmid m} \GL_2(\mathbb{Z}/\ell\mathbb{Z}).
$$
The result now follows by applying the snake lemma to the commutative diagram
$$
\begin{CD}
1 @>>> G(m) @>>> \Gal(L/\mathbb{Q}) @>>> \displaystyle{\prod_{\ell \nmid m}  \GL_2(\mathbb{Z}/\ell\mathbb{Z})} @>>> 1 \\
@.            @VVV                  @VVV                        						@VVV				\\
1 @>>> \displaystyle{\prod_{\ell \mid m} G(\ell)} @>>> \displaystyle{\prod_\ell G(\ell)} @>>> \displaystyle{\prod_{\ell \nmid m}  \GL_2(\mathbb{Z}/\ell\mathbb{Z})} @>>> 1.
\end{CD}
$$
\end{proof}
\end{comment}

%////////////// END COMMENT //////////////////////

\section{Cyclic reduction of elliptic curves}\label{CyclicRedEC}
In this section we consider an elliptic curve analogue of Artin's classical conjecture on primitive roots. Recall that this conjecture predicts the density of primes $p$ such that a given rational number is a primitive root modulo $p$. In \cite{LTArtin}, Lang and Trotter formulated an analogous conjecture for elliptic curves over $\mathbb{Q}$. Namely, if $P$ is a point of $E(\mathbb{Q})$ of infinite order, then the problem is to determine the density of primes $p$ for which $\tilde{E}(\mathbb{F}_p)$ is generated by $\tilde{P}$, the reduction of $P$ modulo $p$.

Note that for there to exist even one prime $p$ of good reduction with this property, a necessary condition is that the group $\tilde{E}(\mathbb{F}_p)$ be cyclic, and that is the question we consider here. In \cite{SerreResume}, Serre showed assuming the Generalized Riemann Hypothesis that the set of primes $p$ such that $\tilde{E}(\mathbb{F}_p)$ is cyclic has a density. He did this by adapting Hooley's argument of conditionally proving Artin's conjecture on primitive roots. Namely, we have the following:

\begin{thm}[Serre, 1976]
Let $E$ be an elliptic curve defined over $\mathbb{Q}$ with conductor $N_E$. Assuming GRH we have that
$$
|\{ p \leqslant x \text{ prime} :  p \nmid N_E,\ \tilde{E}(\mathbb{F}_p) \text{ is cyclic} \}| \sim C_E \frac{x}{\log{x}}
$$
as $x\rightarrow \infty$, where $\displaystyle{C_E:=\sum_{n\geqslant 1} \frac{\mu(n)}{[\mathbb{Q}(E[n]):\mathbb{Q}]}}$. 
\end{thm}
We explicitly evaluate this density $C_E$ as an Euler product. Note that the condition of $\tilde{E}(\mathbb{F}_p)$ being cyclic is completely determined by $\rho_E(G_\mathbb{Q})$. Indeed, $\tilde{E}(\mathbb{F}_p)$ is cyclic if and only if $p$ does not split completely in the field $\mathbb{Q}(E[\ell])$ for any $\ell \neq p$. Note that this condition is automatically satisfied when $\ell > p$, since $p$ splitting completely in $\mathbb{Q}(E[\ell])$ implies $p\equiv 1 \pmod{\ell}$. In other words, if for each prime $\ell$ we define the set $S(\ell) := G(\ell)-\{1\}$, then for all $p \nmid N_E$ the group $\tilde{E}(\mathbb{F}_p)$ is cyclic if and only if $\rho_{\ell}(\Frob_p)\in S(\ell)$ for any $\ell < p$, i.e. if $p$ does not split completely in $\mathbb{Q}(E[\ell])$.

By the Chebotarev density theorem, the set of primes $p$ that do not split completely in $\mathbb{Q}(E[\ell])$ has density equal to 
$$
\delta_\ell := \frac{|S(\ell)|}{|G(\ell)|} = 1-\frac{1}{[\mathbb{Q}(E[\ell]):\mathbb{Q}]}.
$$
If we assume that the various splitting conditions at each prime $\ell$ are independent, then it is reasonable to expect that the density of primes $p$ for which $\tilde{E}(\mathbb{F}_p)$ is cyclic is equal to $\prod_\ell \delta_\ell$. However as we know, this assumption of independence is not correct, as different torsion fields may have non-trivial intersection. To be precise, for each square-free integer $d$ let 
$$
\mathcal{S}_d:= \prod_{\ell \mid d} S(\ell), \quad \mathcal{G}_d:= \prod_{\ell \mid d} G(\ell).
$$
By Chebotarev, the density of primes $p$ such that $p\nmid N_E$ and $\rho_\ell(\Frob_p) \in S(\ell)$ for all $\ell\mid d$ and $\ell\neq p$ is equal to $|\mathcal{S}_d\cap G(d)|/|G(d)|$. If we let $d$ increase to infinity ranging over square-free integers, then Serre's above result implies that, assuming GRH,
\begin{equation}\label{CELimit}
C_E = \lim_{d\rightarrow \infty}\frac{|\mathcal{S}_d\cap G(d)|}{|G(d)|}
\end{equation}
where the limit will be seen to exist.

Now let $m=\prod_{\ell\mid m_E} \ell$ be the square-free part of $m_E$, and let $d$ be a square-free integer coprime to $m$. By (\ref{ConstantChi}) we have
$$
\frac{|\mathcal{S}_{md}\cap G(md)|}{|G(md)|} = \frac{|\mathcal{S}_m\cap G(m)|}{|G(m)|} \prod_{\ell \mid d} \frac{|S(\ell)|}{|G(\ell)|}.
$$
For $\ell$ coprime to $m_E$, we have that $|S(\ell)|/|G(\ell)|$ is $1+\mathcal{O}(1/\ell^4)$ so the limit in (\ref{CELimit}) does indeed exist. Letting $d$ tend to infinity over the square-free numbers then gives
$$
C_E = \frac{|\mathcal{S}_m\cap G(m)|}{|G(m)|} \prod_{\ell \nmid m} \frac{|S(\ell)|}{|G(\ell)|}.
$$
The above discussion implies that if we do take into account entanglements, then assuming GRH we have
\begin{equation}\label{CorrectionAsQuotient}
C_E = \mathfrak{C}_E \prod_\ell \delta_\ell
\end{equation}
where $\mathfrak{C}_E$ is an \emph{entanglement correction factor}, and explicitly evaluating such densities amounts to computing the correction factors $\mathfrak{C}_E$. The entanglement correction factor $\mathfrak{C}_E$ arises as the factor by which $C_E$ differs from the uncorrected value $\lim_{d\rightarrow \infty}|\mathcal{S}_d|/|\mathcal{G}_d|=\prod_\ell \delta_\ell$. We will use Theorem \ref{ThmGoodFrobeniusCalc} for evaluating $\mathfrak{C}_E$ as a character sum for elliptic curves with abelian entanglements.

%/////////////////////////////COMMENT////////////////////////////////

\begin{comment}
The following lemma tells us that $\Phi_{m_E}$ measures the full extent to which the distinct torsion fields of $E$ have any entanglements. In particular $\Phi_{m_E}$ is always of order at least $2$.
\begin{lem}\label{squareFreeM}
Let $m$ be a positive integer and $d$ be a positive integer coprime to $m_E$ . Then $\Phi_{md} \simeq \Phi_{m}$.
\end{lem}

\begin{proof}
Again there is a map $\psi_{md}$ and an abelian group $\Phi_{md}$ which fit into the short exact sequence
$$
1 \longrightarrow G(md) \longrightarrow \prod_{\ell \mid md} G(\ell) \xrightarrow{\phantom{a}\psi_{md}\phantom{b}} \Phi_{md} \longrightarrow 1.
$$
As $d$ is coprime to $m_E$, by Serre's open image Theorem we have that
\begin{equation}\label{ConstantChi}
G(md) = G(m) \times \prod_{\ell \mid d} G(\ell)
\end{equation}
It follows that $G(\ell)$ is contained in the kernel of $\psi_{md}$ for any $\ell \mid d$, and in follows that $\Phi_{md}\simeq \Phi_m$.
\end{proof}
\end{comment}

%///////////////////////////////END COMMENT//////////////////////

\begin{thm}\label{ThmCorrectionFactor}
Assume $E/\mathbb{Q}$ has abelian entanglements, and let $\Phi_m$ be as in (\ref{ExSeqPsi}). Let $\tilde{\chi}\in \widehat{\Phi}_m$ be a character of $\Phi_m$ and let $\chi$ be the character of $\mathcal{G}_m$ obtained by composing $\tilde{\chi}$ with $\psi_m$. Define $E_{\chi,\ell}$ by
$$
E_{\chi,\ell}=
\begin{cases}
1 & \text{if $\chi$ is trivial on $G(\ell)$,} \\
\frac{-1}{[\mathbb{Q}(E[\ell]):\mathbb{Q}]-1} & \text{otherwise.}
\end{cases}
$$
Then
$$
C_E = \mathfrak{C}_E \prod_{\ell} \delta_\ell
$$
where the entanglement correction factor $\mathfrak{C}_E$ is given by
$$
\mathfrak{C}_E = 1+\sum_{\tilde{\chi}\in \widehat{\Phi}-\{1\}}\prod_{\ell | m} E_{\chi,\ell}.
$$
\end{thm}

\begin{proof}
By Theorem \ref{ThmGoodFrobeniusCalc} we have that
\begin{equation*}
\frac{|\mathcal{S}_m\cap G(m)|}{|G(m)|} = \frac{|\mathcal{S}_m|}{|\mathcal{G}_m|} \Bigg ( 1 + \sum_{\tilde{\chi} \in \widehat{\Phi}\backslash\{1\}} \prod_{\ell \mid m} E_{\chi,\ell} \Bigg ),
\end{equation*}
where $E_{\tilde{\chi},\ell} $ is the average value of $\chi_\ell$ on $S(\ell)$. By (\ref{CorrectionAsQuotient}), we know that
\begin{align*}
\mathfrak{C}_E &= \frac{C_E}{\prod_\ell \delta_\ell} \\
       	 			     &= \frac{|\mathcal{S}_m\cap G(m)|/|G(m)|}{|\mathcal{S}_m|/|\mathcal{G}_m|}.
\end{align*}
Finally, notice that if $\chi$ is non-trivial on $G(\ell)$ then $\chi_\ell$ is non-trivial, hence
$$
\sum_{x\in S(\ell)} \chi_\ell (x)  = \bigg( \sum_{x\in G(\ell)} \chi_\ell (x)\bigg) - \chi_\ell (1) = -1.
$$
This completes the proof.
\end{proof}

\begin{rem}
Note that in the above theorem we may replace $m$ by any square-free multiple of it. Indeed, for any $\tilde{\chi}$, it follows from Lemma \ref{squareFreeM} that $E_{\chi,\ell}=1$ for any $\ell \nmid m$, hence the product $\prod_{\ell | m} E_{\chi,\ell}$ does not change, and the quotient of $|\mathcal{S}_{md}\cap G(md)|/|G(md)|$ and $|\mathcal{S}_{md}|/|\mathcal{G}_{md}|$ is constant as $d$ tends to infinity.
\end{rem}

In what follows we will use Theorem \ref{ThmCorrectionFactor} to compute $\mathfrak{C}_E$ for various elliptic curves over $\mathbb{Q}$. 

\subsection{Serre curves}\label{SerreEntanglement}
Consider the representation $\rho_E: G_\mathbb{Q} \rightarrow \GL_2(\widehat{\mathbb{Z}})$ given by the action of $G_\mathbb{Q}$ on $E(\overline{\mathbb{Q}})_{\text{tors}}$. Serre has shown in \cite{SerrePG} that the image of $\rho_E$ is always contained in a specific index $2$ subgroup of $\GL_2(\widehat{\mathbb{Z}})$ and thus $\rho_E$ is \emph{never} surjective. Following Lang and Trotter, we define an elliptic curve $E$ over $\mathbb{Q}$ to be a Serre curve if $[\GL_2(\widehat{\mathbb{Z}}):G]=2$.

It follows from the result of Serre that Serre curves are elliptic curves over $\mathbb{Q}$ whose Galois action on their torsion points is as large as possible. Jones has shown in \cite{JonesSerreCurves} that ``most'' elliptic curves over $\mathbb{Q}$ are Serre curves (see Section FIX THIS for the more precise statement) . Thus they are prevalent over $\mathbb{Q}$ and we also have complete understanding of their Galois theory, and this makes their entanglement factors particularly easy to handle in conjunction with Theorem \ref{ThmCorrectionFactor}.

First we briefly describe the index $2$ subgroup $H_E$ of $\GL_2(\widehat{\mathbb{Z}})$ (see \cite{SerrePG}, page 311 for more details). To this end let $\chi_\Delta:G_\mathbb{Q}\rightarrow \{\pm 1 \}$ be the character associated to $K:=\mathbb{Q}(\sqrt{\Delta})$, where $\Delta$ is the discriminant of any Weierstrass model of $E$ over $\mathbb{Q}$, and note that $\chi_\Delta$ does not depend on the choice of model. Let
$$
\varepsilon:\GL_2(\mathbb{Z}/2\mathbb{Z}) \longrightarrow \{\pm 1\}
$$
be the signature map under any isomorphism $GL_2(\mathbb{Z}/2\mathbb{Z}) \simeq S_3$. Then as $K\subset \mathbb{Q}(E[2])$, one can check that $\chi_\Delta = \varepsilon \circ \rho_{E,2}$. 

Note that $K \subset \mathbb{Q}(\zeta_{|D|})$, where $D$ is the discriminant of $\mathbb{Q}(\sqrt{\Delta})$. Then there exists a unique quadratic character $\alpha: (\mathbb{Z}/|D|\mathbb{Z})^\times \rightarrow \{\pm 1 \}$ such that $\chi_\Delta = \alpha \circ \det \rho_{E,|D|}$. From this it follows that $\varepsilon \circ \rho_{E,2} = \alpha \circ \rho_{E,|D|}$. If we then define $M_E = \text{lcm} (|D|,2)$ and
$$
H_{M_E} := \big \{ A\in \GL_2(\mathbb{Z}/M_E\mathbb{Z})\ :\ \varepsilon(A\text{ mod }2) = \alpha\big(\det(A\text{ mod }|D|)\big)\big \},
$$
then it follows from the above discussion that $H_{M_E}$ contains $G(M_E)$. If we let $H_E$ be the inverse image of $H_{M_E}$ in $\GL_2(\widehat{\mathbb{Z}})$ under the reduction map, then $H_E$ is clearly an index $2$ subgroup of $\GL_2(\widehat{\mathbb{Z}})$ which contains $G$. We have then that $G$ is a Serre curve if and only if $\rho_E(G_\mathbb{Q}) = H_E$. It follows from the above discussion that all Serre curves have abelian entanglements.

\begin{pro}\label{SerreCurveFirstCase}
Let $E/\mathbb{Q}$ be a Serre curve. Let $D$ be the discriminant of $\mathbb{Q}(\sqrt{\Delta})$ where $\Delta$ is the discriminant of any Weierstrass model of $E$ over $\mathbb{Q}$. Then
$$
C_E = \mathfrak{C}_E \prod_{\ell}\bigg(1-\frac{1}{(\ell^2-1)(\ell^2-\ell)} \bigg )
$$
where the entanglement correction factor $\mathfrak{C}_E$ is given by
$$
\mathfrak{C}_E =
\begin{cases}
1 & \text{ if $D\equiv 0 \pmod{4}$ } \\
1+ \displaystyle{\prod_{\ell \mid 2D} \frac{-1}{(\ell^2-1)(\ell^2-\ell)-1}} & \text{ if $D\equiv 1 \pmod{4}$ }
\end{cases}
$$
\end{pro}

\begin{proof}
Since $E$ is a Serre curve, we have that $G(\ell) = \GL_2(\mathbb{Z}/\ell \mathbb{Z})$ holds for all $\ell$, hence $[\mathbb{Q}(E[\ell]):\mathbb{Q}] = (\ell^2-1)(\ell^2-\ell)$.

Now suppose first that $D \equiv 0 \pmod{4}$. Then $m_E = |D|$ is divisible by $4$, hence we have that
$$
G(m) = \prod_{\ell \mid m} G(\ell)
$$
for all square-free $m$. It follows that $\Phi_m\simeq \{1\}$ hence its character group is trivial and $\mathfrak{C}_E =1$.

Now suppose $D \equiv 1 \pmod{4}$. In this case $m_E = 2|D|$ is square-free, hence $G(m_E)$ is an index $2$ subgroup of $\prod_{\ell \mid m_E}G(\ell)$ and $\Phi\simeq \{\pm 1\}$. For each $\ell>2$ dividing $m_E$, $\chi_\ell$ is the character given by the composition $G(\ell) \xrightarrow{\det} \big(\mathbb{Z}/\ell \mathbb{Z}\big)^* \rightarrow \{\pm 1\}$, that is $\chi_\ell = \blegendre{\det}{\ell}$, and $\chi_2:=\varepsilon$ is the signature map under an isomorphism $\GL_2(\mathbb{Z}/2\mathbb{Z})\simeq S_3$. If we let $\chi:=\prod_{\ell \mid m_E} \chi_\ell$ then we have an exact sequence
$$
1\longrightarrow G(m_E) \longrightarrow \prod_{\ell \mid m_E} G(\ell) \xrightarrow{\phantom{a}\chi\phantom{b}} \{\pm 1\} \longrightarrow 1.
$$
Clearly each $\chi_\ell$ is non-trivial on $G(\ell)$ for each $\ell$ dividing $m_E$ so the result follows from Theorem \ref{ThmCorrectionFactor} and using that $\Phi_{m_E} \simeq \{\pm 1\}$.
\end{proof}

\subsection{Example: $Y^2 + Y = X^3-X^2-10X-20$}

We now consider the elliptic curve over $\mathbb{Q}$ defined by the Weierstrass equation $Y^2 + Y = X^3-X^2-10X-20$. The Galois theory for this elliptic curve has been worked out by Lang and Trotter in \cite{LT}, and in particular they have shown that $m_E = 2 \cdot 5^2 \cdot 11$, and that the following properties hold:

\begin{itemize}
\item
$G(2) = \GL_2(\mathbb{Z}/2\mathbb{Z})$.
\item
$E$ has a rational $5$-torsion point, and $\mathbb{Q}(E[5]) = \mathbb{Q}(\zeta_5)$.
\item
$[\mathbb{Q}(E[5^2]):\mathbb{Q}(E[5])] = 5^4$, hence $5$ is stable.
\item
$\mathbb{Q}(E[5^2]) \cap \mathbb{Q}(E[11]) = \mathbb{Q}(\zeta_{11})^+$,
where $\mathbb{Q}(\zeta_{11})^+$ is the real quadratic subfield of $\mathbb{Q}(\zeta_{11})$. This implies there is a map
$$
\phi_5: G(5^2) \longrightarrow \big( \mathbb{Z}/ 11\mathbb{Z} \big)^\times / \{\pm 1\}.
$$
We make this map explicit. There is a basis for $E[5^2]$ over $\mathbb{Z}/25\mathbb{Z}$ under which we have
$$
G(5^2) = \left \{ \begin{pmatrix} 1+5a & 5b \\ 5c & u \end{pmatrix} : a,b,c,d \in \mathbb{Z}/25\mathbb{Z},\ u\in \big(\mathbb{Z}/25\mathbb{Z} \big)^\times \right \}.
$$
Define the (surjective) homomorphism 
\begin{align*}
\psi : G(5^2) &\longrightarrow \mathbb{Z}/5\mathbb{Z} \\
        \begin{pmatrix} 1+5a & 5b \\ 5c & u \end{pmatrix} & \longmapsto a \mod 5.
\end{align*}
Then $\phi_5$ is given by
$$
A \longmapsto (\pm 2)^{\psi(A)},
$$
where we note that $\pm 2$ is a generator of $(\mathbb{Z}/11\mathbb{Z})^\times / \{ \pm 1 \}$.
\item
$\mathbb{Q}(E[2]) \cap \mathbb{Q}(E[11]) = \mathbb{Q}(\sqrt{-11})$. 
\end{itemize}
From this we conclude that $E$ has abelian entanglements and
\begin{multline*}
G(2\cdot 5^2\cdot 11) = \Big\{ (g_2,g_{25}, g_{11}) \in G(2) \times G(5^2) \times G(11) : \\ \varepsilon(g_2) = \blegendre{\det(g_{11})}{11},\ \phi_5(g_5) = \phi_{11}(g_{11}) \Big\}.
\end{multline*}

\begin{pro}
Let $E/\mathbb{Q}$ be the elliptic curve given by Weierstrass equation $Y^2+Y= X^3-X^2-10X-20$. Then we have
\begin{align*}
C_E &= \frac{3}{4} \mathfrak{C}_E \prod_{\ell \neq 5} \bigg(1 - \frac{1}{(\ell^2-\ell)(\ell^2-1)} \bigg) \\
	   &\approx 0.611597,
\end{align*}
where $\mathfrak{C}_E$ is given by
$$
\mathfrak{C}_E = 1 + \frac{1}{65995}.
$$
\end{pro}

\begin{proof}
As before we take $m = 2\cdot 5\cdot 11$ to be the square-free part of $m_E$. Because $E$ has abelian entanglements there is an exact sequence
$$
1 \longrightarrow G(2\cdot 5\cdot 11) \longrightarrow G(2) \times G(5) \times G(11) \xrightarrow{\phantom{a}\chi\phantom{b}} \Phi_{110} \longrightarrow 1
$$
From the description of $G(2\cdot 5^2\cdot 11)$ it follows that $G(2\cdot 5\cdot 11) = G(22) \times G(5)$, hence $\Phi_{110} \simeq \{ \pm 1\}$. It follows that if we set $\chi_2$ equal to the sign character $\varepsilon$, $\chi_{11}$ to $\blegendre{\det(g_{11})}{11}$ and $\chi_5$ be trivial, then $\chi = \chi_2\chi_5\chi_{11}$.

By Theorem \ref{ThmCorrectionFactor} we have
$$
C_E = \mathfrak{C}_E  \prod_\ell \delta_\ell.
$$
where
$$
\mathfrak{C}_E = 1 + E_{\chi_2} E_{\chi_5} E_{\chi_{11}}.
$$
From the description of $G(\ell)$ it is then straightforward to compute $\delta_\ell$ as well as $E_{\chi_\ell}$ for every $\ell$.
\end{proof}

\begin{rem}
Note that in this example, even though the Galois theory of $E$ was considerably more complicated than that of a Serre curve, at the `square-free' torsion level it was still very similar. Indeed, the subgroup $G(110) \leqslant G(2) \times G(5) \times G(11)$ was still cut out only by a quadratic character.
\end{rem}

\section{Cyclic reduction for primes in an arithmetic progression}\label{CyclicRedAP}
We now consider a variant of the problem on cyclic reduction of elliptic curves. We have been looking at the density of primes $p$ for which the reduction $\tilde{E}(\mathbb{F}_p)$ is cyclic. Here we impose the additional requirement that $p$ lie in a prescribed residue class modulo some integer $f$. This is just one of many possible generalizations one could consider, and in many of them one should still obtain a density assuming GRH. One of the difficulties that arises however, is the explicit computation of the density as an Euler product. The character sum method we have given allows us to do this in a relatively simple manner.

If we keep the same setup as in Theorem \ref{ThmCorrectionFactor}, then note that the condition we are imposing on $p$ being satisfied is again completely determined by $\rho_E(G_\mathbb{Q})$. In this case however, it is not necessarily enough to consider only the `square-free' torsion fields $\mathbb{Q}(E[\ell])$. Suppose then that we are interested in primes $p$ such that 
\begin{itemize}\label{ArtinConditions}
\item[(i)] 
$\tilde{E}(\mathbb{F}_p)$ is cyclic,
\item[(ii)]
$p\equiv a \pmod{f}$.
\end{itemize}
For each prime power $\ell^\alpha$, define
\begin{align*}
\mathcal{D}_a(\ell^\alpha) := \{A\in \GL_2(\mathbb{Z}/\ell^\alpha\mathbb{Z}) :\ \det{A} \equiv a \pmod{\ell^\alpha} \}, \\
\big(I+\ell M_2(\mathbb{Z}/\ell^\alpha \mathbb{Z}) \big)^c := \{A\in \GL_2(\mathbb{Z}/\ell^\alpha \mathbb{Z}) : A \not \equiv I \pmod{\ell} \}.
\end{align*}
Let $f=\prod_\ell \ell^{e_\ell}$ be the prime factorisation of $f$, and for each $\ell\mid f$ set
\begin{align*}
\Psi_a(\ell^{e_\ell}) :&=\mathcal{D}_a(\ell^{e_\ell}) \cap \big(I+\ell M_2(\mathbb{Z}/\ell^{e_\ell} \mathbb{Z}) \big)^c \\
							   &= \{A\in \GL_2(\mathbb{Z}/\ell^{e_\ell}\mathbb{Z}) : A \not \equiv I \pmod{\ell},\ \det{A} \equiv a \pmod{\ell^{e_\ell}} \}.
\end{align*}
Then set 
$$
S(\ell) := G(\ell^{e_\ell}) \cap \Psi_a(\ell^{e_\ell})
$$
for those $\ell$ dividing $f$, and just as in the case of the previous subsection, set $S(\ell):=G(\ell)-\{1\}$ for all other $\ell$. Then it follows that $p \nmid N_E$ satisfies conditions (i) and (ii) above if and only if for any $\ell \nmid p$ one has
\begin{itemize}
\item[(i)]
$\rho_\ell (\Frob_p) \in S(\ell)$ if $\ell \nmid f$,
\item[(ii)]
$\rho_{\ell^{e_\ell}}(\Frob_p) \in S(\ell)$ if $\ell \mid f$.
\end{itemize}
Then the density of $p$ having the `right' local behaviour at $\ell$ equals
$$
\delta_\ell = \begin{cases}
|S(\ell)|/|G(\ell)| & \text{ if $\ell \nmid f$} \\
|S(\ell)|/|G(\ell^{e_\ell})| & \text{ if $\ell \mid f$}
\end{cases}
$$
and the naive density of primes satisfying conditions (i) and (ii) equals $\prod_\ell \delta_\ell$.

To account for entanglements, we proceed more or less along the same line as the case without the condition of $p$ lying in a prescribed residue class, with some slight modifications. That is, let
$$
m:= \prod_{\ell \mid (f,m_E)} \ell^{e_\ell} \prod_{\substack{\ell \mid m_E \\ \ell \nmid f}} \ell
$$
For any square-free $d$ coprime to $m$, define
$$
\mathcal{S}_{md} := \prod_{\ell \mid md} S(\ell), \quad \mathcal{G}_{md} := \prod_{\ell \mid (f,m)} G(\ell^{e_\ell}) \prod_{\substack{\ell \mid md \\ \ell \nmid f}} G(\ell).
$$
By Corollary \ref{CorEquivEntang}
$$
G(md) \leqslant \mathcal{G}_{md}
$$
has abelian entanglements, hence we have an exact sequence
$$
1 \longrightarrow G(md) \longrightarrow \mathcal{G}_{md} \xrightarrow{\phantom{a}\psi_{md}\phantom{b}} \Phi_{md} \longrightarrow 1
$$
for some abelian group $\Phi_{md}$. We again have by (\ref{ConstantChi}) that $\Phi_{md} \simeq \Phi_m$ for any square-free $d$ coprime to $m$, and the density we are looking for is then
$$
C_E(a,f) = \lim_{d\rightarrow \infty} \frac{|\mathcal{S}_{md}\cap G(md)|}{|G(md)|} = \frac{|\mathcal{S}_m\cap G(m)|}{|G(m)|} \prod_{\ell \nmid m} \frac{|S(\ell)|}{|G(\ell)|}.
$$

\begin{thm}\label{ThmCorrectionFactorProg}
Let $\tilde{\chi}\in \widehat{\Phi}_m$ be a character of $\Phi_m$ and let $\chi$ be the character of $\mathcal{G}_m$ obtained by composing $\tilde{\chi}$ with $\psi_m$. Define $E_{\chi,\ell}$ by
$$
E_{\tilde{\chi},\ell} = \sum_{x\in S(\ell)} \frac{\chi_\ell(x)}{|S(\ell)|}.
$$
Then
$$
C_E(a,f) = \mathfrak{C}_E(a,f) \prod_{\ell} \delta_\ell
$$
where the entanglement correction factor $\mathfrak{C}_E(a,f)$ is given by
$$
\mathfrak{C}_E(a,f) = 1+\sum_{\tilde{\chi}\in \widehat{\Phi}_m-\{1\}}\prod_{\ell | m} E_{\chi,\ell}.
$$
\end{thm}

\begin{proof}
The proof is exactly as that of Theorem \ref{ThmGoodFrobeniusCalc} with the obvious modifications.
\end{proof}

It follows from the previous theorem that in order to evaluate the correction factors $\mathfrak{C}_E(a,f)$ it suffices to compute the order of $S(\ell)$ as well as the average value of the $\chi_\ell$ on $S(\ell)$.

\subsection{Serre curves}
In what follows we again consider the example of Serre curves. To simplify the following proofs we will henceforth assume $a$ and $f$ are coprime integers. If not, then for a prime $\ell$ dividing $(a,f)$ we obtain $|\Psi_a(\ell^{e_\ell})|=0$ hence $|S(\ell)|=0$ and $C_E(a,f)=0$, which we take to mean the conditions imposed are satisfied for only finitely many $p$.

\begin{lem}\label{deltaEll}
Let $E/\mathbb{Q}$ be a Serre curve, and let $a$ and $f$ be coprime positive integers. Let $D$ be the discriminant of $\mathbb{Q}(\sqrt{\Delta})$ where $\Delta$ is the discriminant of any Weierstrass model of $E$ over $\mathbb{Q}$. Suppose that $|D| \neq 4,8$. Then
$$
\delta_\ell = 
\begin{cases}
\frac{1}{\phi(\ell^{e_\ell})} & \text{ if $a\not\equiv 1 \pmod{\ell}$ and $\ell \mid f$ } \\
\frac{1}{\phi(\ell^{e_\ell})}\Big(1-\frac{1}{\ell(\ell-1)(\ell+1)}\Big) & \text{ if $a \equiv 1 \pmod{\ell}$ and $\ell \mid f$ } \\
1-\frac{1}{(\ell^2-1)(\ell^2-\ell)} & \text{ if $\ell \nmid f$. }
\end{cases}
$$
\end{lem}

\begin{comment}
$$
\prod_\ell \delta_\ell = \frac{1}{\phi(f)} \prod_{\ell \mid (a-1,f)} \bigg( 1- \frac{1}{\ell(\ell-1)(\ell+1)} \bigg) \prod_{\ell \nmid f} \bigg( 1-\frac{1}{(\ell^2-1)(\ell^2-\ell)} \bigg). 
$$
\end{comment}

\begin{proof}
If $\ell \nmid f$ then as before we obtain the local density $\delta_\ell = 1-1/(\ell^2-1)(\ell^2-\ell)$. At $\ell \mid f$ we consider the two cases. If $a \not \equiv 1 \pmod{\ell}$ then 
$$
S(\ell) = \mathcal{D}_a(\ell^{e_\ell})
$$
since any element with determinant $a \not\equiv 1$ cannot be trivial mod $\ell$.  It follows that for such $\ell$ one has $\delta_\ell = 1/\phi(\ell^{e_\ell})$. If $a\equiv 1 \pmod{\ell}$ then we need to count the fraction of elements of $\mathcal{D}_a(\ell^{e_\ell})$ which are non-trivial mod $\ell$. There is a surjective map $G(\ell) \rightarrow \big(\mathbb{Z}/\ell\mathbb{Z}\big)^*$ of degree $\ell(\ell-1)(\ell+1)$, and $\mathbb{Q}(E[\ell])\cap \mathbb{Q}(\zeta_{\ell^{e_\ell}}) = \mathbb{Q}(\zeta_\ell)$ (since $|D|\neq 4,8$) so it follows that this fraction is precisely $1-1/\ell(\ell-1)(\ell+1)$, as desired.
\end{proof}

\begin{lem}\label{Qi}
Let $E$, $a$ and $f$ be as in Lemma \ref{deltaEll}. Suppose further that $|D| = 4$. Then
$$
\delta_2 = 
\begin{cases}
\frac{1}{\phi(2^{e_2})} & \text{ if $a\equiv 3 \pmod{4}$ and $4\mid f$} \\
\frac{1}{\phi(2^{e_2})} \big( 1-\frac{1}{3} \big) & \text{ if $a\equiv 1 \pmod{4}$ and $4\mid f$} \\
\frac{5}{6} & \text{ if $4\nmid f$}.
\end{cases}
$$
\end{lem}

\begin{proof}
The assumption on $D$ implies that $\mathbb{Q}(\sqrt{\Delta}) = \mathbb{Q}(i)$ and $m_E = 4$. Recall that $2^{e_2} || f$ is the highest power of $2$ dividing $f$.  If $e_2 \geqslant 2$ then $a$ is odd, hence is $1$ or $3$ mod $4$. Note that $\mathbb{Q}(\zeta_{2^{e_2}})\cap \mathbb{Q}(E[2]) = \mathbb{Q}(i)$. Now the fraction of elements $A\in G(2^{e_2})$ such that $A\in \mathcal{D}_a(2^{e_2})$ equals $1/\phi(2^{e_2})$. If $a\equiv 3 \pmod{4}$ then any such $A\in \mathcal{D}_a(2^{e_2})$ acts non-trivially on $\mathbb{Q}(i)$, hence is non-trivial mod $2$. It follows that $S(2) = \mathcal{D}_a(2^{e_2})$ and $\delta_2 = 1/\phi(2^{e_2})$. If $a\equiv 1 \pmod{4}$, then because $[\mathbb{Q}(E[2]) : \mathbb{Q}(i)] =3$ exactly $1-1/3$ of the elements in $A\in \mathcal{D}_a(2^{e_2})$ are in $S(2)$. Finally suppose $e_2 < 2$. Then the only condition at $2$ is being non-trivial mod $2$, and the conclusion follows.
\end{proof}

\begin{lem}\label{sqrttwo}
Let $E$, $a$ and $f$ be as in Lemma \ref{deltaEll}. Suppose further that $|D| = 8$. Then
\begin{itemize}
\item[(i)] If $\mathbb{Q}(\sqrt{\Delta}) = \mathbb{Q}(\sqrt{2})$ then
$$
\delta_2 = 
\begin{cases}
\frac{1}{\phi(2^{e_2})} & \text{ if $a\equiv 3$ or $ 5 \pmod{8}$ and $8\mid f$} \\
\frac{1}{\phi(2^{e_2})} \big( 1-\frac{1}{3} \big) & \text{ if $a\equiv 1$ or $7 \pmod{8}$ and $8\mid f$} \\
\frac{5}{6} & \text{ if $8\nmid f$}.
\end{cases}
$$
\item[(ii)]  $\mathbb{Q}(\sqrt{\Delta}) = \mathbb{Q}(\sqrt{-2})$ then
$$
\delta_2 = 
\begin{cases}
\frac{1}{\phi(2^{e_2})} & \text{ if $a\equiv 5$ or $ 7 \pmod{8}$ and $8\mid f$} \\
\frac{1}{\phi(2^{e_2})} \big( 1-\frac{1}{3} \big) & \text{ if $a\equiv 1$ or $3 \pmod{8}$ and $8\mid f$} \\
\frac{5}{6} & \text{ if $8\nmid f$}.
\end{cases}
$$
\end{itemize}
\end{lem}

\begin{proof}
We proceed similarly to Lemma \ref{Qi}. The assumption on $D$ implies that $\mathbb{Q}(\sqrt{\Delta}) = \mathbb{Q}(\sqrt{\pm 2})$. If $e_2\geqslant 3$ then in this case $\mathbb{Q}(\zeta_{2^{e_2}})\cap \mathbb{Q}(E[2]) = \mathbb{Q}(\sqrt{\pm 2})$. In case (i), elements in $\mathcal{D}_a(2^{e_2})$ act non-trivially on $\mathbb{Q}(\sqrt{2})$ if and only if $a\equiv 3$ or $5 \pmod{8}$, hence the conclusion. Case (ii) follows from the same argument.
\end{proof}

In what remains of this section we will deduce the correction factor $\mathfrak{C}_E(a,f)$. In the following lemmas we compute the local factors $E_\ell$ for the different primes $\ell$ dividing $m_E$. As is often the case, the prime $2$ requires special consideration and we split the computation of the local correction factor $E_2$ into various cases. Keep the same notation for $E,a,f$ and $D$, and suppose further that $|D|\neq 4,8$. Then $m_E$ contains at least one odd prime factor and we have an exact sequence
$$
1 \longrightarrow G(m) \longrightarrow \prod_{\ell \mid (f,m_E)} G(\ell^{e_\ell}) \prod_{\substack{\ell \mid m_E \\ \ell \nmid f}} G(\ell)  \xrightarrow{\phantom{a}\chi\phantom{b}} \{\pm 1\} \longrightarrow 1
$$
where $\chi = \prod_\ell \chi_\ell$ is a product of characters $\chi_\ell$. Here $\chi_\ell$ is given by the composition $G(\ell^{e_\ell}) \rightarrow G(\ell) \xrightarrow{\det} \big(\mathbb{Z}/\ell \mathbb{Z}\big)^* \rightarrow \{\pm 1\}$ and $\chi_2$ is the character corresponding to the quadratic extension $\mathbb{Q}(E[2^{\alpha_2}])\cap \mathbb{Q}(E[m/2^{\alpha_2}])$, where $2^{\alpha_2} || m$. When $e_2=1$ for instance, $\chi_2$ is the signature map $\GL_2(\mathbb{Z}/2\mathbb{Z})\rightarrow \{\pm 1\}$, corresponding to the quadratic extension $\mathbb{Q}(\sqrt{\Delta})$.

\begin{lem}\label{OneModFour}
Suppose $\ord_2(D)=0$. Then $E_2 = -1/5$.
\end{lem}

\begin{proof}
Since $D\equiv 1 \pmod{4}$ it follows that $m_E=2|D|$ and $\chi_2$ is the signature map. Let $2^{e_2} || f$ be the largest power of $2$ dividing $f$. If $e_2\leqslant 0$ then $E_2 = -1/5$ by the same argument as in Proposition \ref{SerreEntanglement}. If $e_2>1$, then $S(2)\subset G(e^{e_2})$ consists of the elements of $\mathcal{D}_a(2^{e_2})$ which are non-trivial mod $2$. 

Because $m_E = 2|D|$ with $D$ odd, $\chi_2$ is the signature map, hence it factors through the surjection $G(2^{e_2}) \rightarrow \Gal(\mathbb{Q}(E[2]), \zeta_{2^{e_2}})$, so we have a commutative diagram
\begin{center}
\begin{tikzpicture}[auto]
\node (a) at (0,2) {$G(2^{e_2})$};
\node (b) at (0,0) {$\Gal(\mathbb{Q}(E[2],\zeta_{2^{e_2}})$};
\node (c) at (2,2) {$\{\pm 1\}$};

\node at (-.3,1.2) {};
\node at (0.7,1.2) {$\chi_2'$};

\draw[->,thick] (a) -- (b);
\draw[->,thick] (a) -- (c);
\draw[->,thick] (b) -- (c);
\end{tikzpicture}.
\end{center}
Let $S'(2)$ be the image of $S(2)$ under the surjection $G(2^{e_2}) \rightarrow \Gal(\mathbb{Q}(E[2]), \zeta_{2^{e_2}})$. Then note that because $\mathbb{Q}(\zeta_{2^{e_2}})\cap\mathbb{Q}(E[2]) = \mathbb{Q}$, for each $\sigma\in G(2)$ there is a unique $\sigma'\in \Gal(\mathbb{Q}(E[2]), \zeta_{2^{e_2}})$ such that $\sigma(\zeta_{2^{e_2}}) = \zeta_{2^{e_2}}^a$ and $\sigma' \equiv \sigma \pmod{2}$. It follows that
$$
\sum_{x\in S'(2)} \chi'(x) = -1
$$
and the conclusion follows.
\end{proof}

\begin{lem}\label{ordDtwo}
Suppose $\ord_2(D) = 2$. We have
\begin{itemize}
\item[(i)]
If $|D|\neq 4$ and $4\mid f$ then 
$$
E_2 = -\blegendre{a}{4} \frac{1}{5}.
$$
\item[(ii)]
If $|D|=4$ or $4\nmid f$ then 
 $$
 E_2=0.
 $$
\end{itemize}
\end{lem}

\begin{proof}
If $4\nmid f$ then because $m_E = |D|$ it follows that $m_E \nmid m$, hence
$$
G(m) = \prod_{\ell \mid (f,m_E)} G(\ell^{e_\ell}) \prod_{\substack{\ell \mid m_E \\ \ell \nmid f}} G(\ell) 
$$
and $\Phi_m \simeq \{1\}$, so $E_2 =0$. Similarly if $|D|=4$ then $m_E$ has no odd prime factors and we again conclude $E_2 = 0$. 

Now suppose $|D| \neq 4$ and $4 \mid f$. If we let $\Delta_{\mathrm{sf}}$ denote the square-free part of $\Delta$, then the assumption on $\ord_2(D)$ implies that $\Delta_{\mathrm{sf}} \equiv 3 \pmod{4}$. Also, because $4\mid f$, we have that $\mathbb{Q}(i)\subset \mathbb{Q}(E[2^{e_2})$, hence
$$
\mathbb{Q}(\sqrt{i\Delta_{\mathrm{sf}}}) = \mathbb{Q}(E[2^{e_2}])\cap \mathbb{Q}(E[m/2^{e_2}])
$$
and $\chi_2$ is the character corresponding to this quadratic extension. If we define
$$
\chi_i:G(2^{e_2})\rightarrow \{\pm 1\}, \quad \chi_\Delta: G(2^{e_2})\rightarrow \{\pm 1\}
$$
to be the characters corresponding to the quadratic extensions $\mathbb{Q}(i)$ and $\mathbb{Q}(\sqrt{\Delta})$, respectively, then $\chi_2=\chi_i\chi_\Delta$. Now $\chi_i$ has constant value equal to $\blegendre{a}{4}$ on $S(2)$, and by the same argument as in Lemma \ref{OneModFour} $\chi_\Delta$ has average value $-1/5$ on $S(2)$. It follows then that
\begin{align*}
E_2 &= \frac{1}{{S(2)}} \sum_{x\in S(2)} \chi_2(x) \\
      &= \frac{1}{{S(2)}} \sum_{x\in S(2)} \chi_i(x)\chi_\Delta(x) \\
      &= -\blegendre{a}{4}\frac{1}{5}.
\end{align*}
\end{proof}

To deal with the case of $\ord_2(D) = 3$, we establish the following notation. Note that if $\ord_2(D) = 3$ then we must have that $2 \mid \Delta_{\mathrm{sf}}$. Let $\Delta'$ be such that $\Delta_{\mathrm{sf}} = 2\Delta'$. 

\begin{lem}\label{ordDthree}
Suppose $\ord_2(D)=3$, and keep the notation above. We have
\begin{itemize}
\item[(i)]
If $|D|\neq 8,\ 8\mid f$ and $\Delta'\equiv 1 \pmod{4}$ then
$$
E_2 = \begin{cases}
1/5 & \text{ if $a \equiv 1$ or $7 \pmod{8}$ } \\
-1/5 & \text{ if $a \equiv 3$ or $5 \pmod{8}$ }.
\end{cases}
$$
\item[(ii)]
If $|D|\neq 8,\ 8\mid f$ and $\Delta'\equiv 3 \pmod{4}$ then
$$
E_2 = \begin{cases}
1/5 & \text{ if $a \equiv 1$ or $3 \pmod{8}$ } \\
-1/5 & \text{ if $a \equiv 5$ or $7 \pmod{8}$ }.
\end{cases}
$$
\item[(iii)]
If $|D|=8$ or $8 \nmid f$ then
$$
E_2 = 0.
$$
\end{itemize}
\end{lem}

\begin{proof}
If $|D|=8$ or $8 \nmid f$ then by the same reasoning as in Lemma \ref{ordDtwo} we conclude $E_2 = 0$. Assume then that $|D| \neq 8$ and $8\mid f$. Because $8\mid f$, we have that $\mathbb{Q}(\sqrt{\pm 2})\subset \mathbb{Q}(E[2^{e_2}])$. Let 
$$
\chi_{\sqrt{2}}:G(2^{e_2})\rightarrow \{\pm 1\}, \quad \chi_{\sqrt{-2}}:G(2^{e_2})\rightarrow \{\pm 1\}, \quad \chi_\Delta: G(2^{e_2})\rightarrow \{\pm 1\}
$$
to be the characters corresponding to the quadratic extensions $\mathbb{Q}(\sqrt{2})$, $\mathbb{Q}(\sqrt{-2})$ and $\mathbb{Q}(\sqrt{\Delta})$, respectively. If $\Delta'\equiv 1 \pmod{4}$ then
$$
\mathbb{Q}(\sqrt{\Delta'}) = \mathbb{Q}(E[2^{e_2}]) \cap \mathbb{Q}(E[m/2^{e_2}])
$$
and $\chi_2$ is the quadratic character corresponding to this extension, with $\chi_2=\chi_{\sqrt{2}}\chi_\Delta$. If $\Delta'\equiv 3 \pmod{4}$ then
$$
\mathbb{Q}(\sqrt{-\Delta'}) = \mathbb{Q}(E[2^{e_2}]) \cap \mathbb{Q}(E[m/2^{e_2}])
$$
and $\chi_2$ is the quadratic character corresponding to this extension, with $\chi_2=\chi_{\sqrt{-2}}\chi_\Delta$. Now note that $\chi_{\sqrt{2}}$ has constant value on $S(2)$ equal to $1$ if $a \equiv 1$ or $7 \pmod{8}$, and $-1$ if $a \equiv 3$ or $5 \pmod{8}$, and $\chi_{\sqrt{-2}}$ has constant value on $S(2)$ equal to $1$ if $a \equiv 1$ or $3 \pmod{8}$, and $-1$ if $a \equiv 5$ or $7 \pmod{8}$ We conclude exactly as in Lemma \ref{ordDtwo}.
\end{proof}

\begin{pro}
Let $E/\mathbb{Q}$ be a Serre curve, and let $a$ and $f$ be coprime positive integers. Let $D$ be the discriminant of $\mathbb{Q}(\sqrt{\Delta})$ where $\Delta$ is the discriminant of any Weierstrass model of $E$ over $\mathbb{Q}$. Suppose that $|D| \neq 4,8$. Then
$$
C_E(a,f) = \mathfrak{C}_E(a,f) \frac{1}{\phi (f)} \prod_{\ell \mid (a-1,f)} \bigg( 1- \frac{1}{\ell(\ell-1)(\ell+1)} \bigg) \prod_{\ell \nmid f} \bigg( 1-\frac{1}{(\ell^2-1)(\ell^2-\ell)} \bigg)
$$
where the entanglement correction factor $\mathfrak{C}_E(a,f)$ is given by 
$$
\mathfrak{C}_E(a,f) =1 + E_2 \prod_{\substack{\ell \mid (D,f) \\ \ell \neq 2}} \blegendre{a}{\ell} \prod_{\substack{\ell \mid D \\ \ell \nmid 2f}} \frac{-1}{(\ell^2-1)(\ell^2-\ell)-1}.
$$
Here $E_2$ is given by Lemmas \ref{OneModFour}, \ref{ordDtwo} and \ref{ordDthree},
\end{pro}

\begin{proof}
Since $|D|\neq 4,8$, the equality involving $C_E(a,f)$ follows from using Lemma \ref{deltaEll} for all $\ell$. The form of the entanglement correction factor at $2$ follows from Lemmas \ref{OneModFour}, \ref{ordDtwo} and \ref{ordDthree}. It remains to consider $\ell \neq 2$. By Theorem $\ref{ThmCorrectionFactorProg}$ if $\ell \nmid f$ and $\ell \mid D$ then $S(\ell) = G(\ell) -\{1\}$ and so 
$$
E_\ell = \frac{-1}{(\ell^2-1)(\ell^2-\ell)-1}.
$$
Ir $\ell\mid (D,f)$ then because $\mathbb{Q}(E[\ell])\cap \mathbb{Q}(\zeta_{\ell^{e_\ell}}) = \mathbb{Q}(\zeta_\ell)$ we have that $\chi_\ell$ has constant value $\blegendre{a}{\ell}$ on $S(\ell)$ and the result follows.
\end{proof}

\begin{comment}
\begin{pro}
Let $E/\mathbb{Q}$ be a Serre curve such that $|D|=4$ or $8$. Then
$$
C_E(a,f) = {1}{\phi (f)} \prod_{\ell \mid (a-1,f)} \bigg( 1- \frac{1}{\ell(\ell-1)(\ell+1)} \bigg) \prod_{\ell \nmid f} \bigg( 1-\frac{1}{(\ell^2-1)(\ell^2-\ell)} \bigg)
$$
\end{pro}
\end{comment}

\begin{cor}
For any $(a,f)$ coprime integers, we have $C_E(a,f)>0$.
\end{cor}

\begin{proof}
It is clear that the naive density $\prod_\ell \delta_\ell $ does not vanish, hence in order for $C_E(a,f)$ to be zero, we would need the correction factor $\mathfrak{C}_E(a,f)$ to be zero, which happens if and only if $\prod_\ell E_\ell = -1$. This is impossible as $E_2$ is always $\pm 1/5$ or $0$.
\end{proof}

\begin{cor}
The correction factor $\mathfrak{C}_E(a,f)$ equals $1$ if and only if $\ord_2(D)>\ord_2(f)$.
\end{cor}

\begin{proof}
From the form of the correction factor it follows that $\mathfrak{C}_E(a,f) = 1 $ if and only if $E_2=0$, and the result follows.
\end{proof}

\subsection{Example: $Y^2 = X^3+X^2+4X+4$}

We look now at an example of a non-Serre curve where the constant $C_E(a,f)$ can vanish. This implies that conjecturally, there should only exist finitely many primes $p$ such that $\tilde{E}(\mathbb{F}_p)$ is cyclic and $p\equiv a \pmod{f}$. Let $E$ be the elliptic curve over $\mathbb{Q}$ given by the Weierstrass equation $Y^2 = X^3+X^2+4X+4$. In \cite{BraSS}, a description of the Galois theory of $E$ is worked out. In particular, for this curve we have that $m_E = 120$, and the following properties hold:
\begin{itemize}
\item
$E$ has a rational $3$-torsion point, and $G(3) \simeq S_3$. 
\item
$E$ has a rational two-torsion point, and $\mathbb{Q}(E[2]) = \mathbb{Q}(i)$.
\item
$G(4)$ has order $16$, and $\mathbb{Q}(E[4]) \cap \mathbb{Q}(E[5]) = \mathbb{Q}(\sqrt{5})$.
\item
$G(8)$ has order $128$, and $\mathbb{Q}(E[8]) \cap \mathbb{Q}(E[5]) = \mathbb{Q}(\zeta_5)$.
\item
$G(5) = \GL_2(\mathbb{Z}/5\mathbb{Z})$ 
\item
$\mathbb{Q}(E[3]) \cap \mathbb{Q}(E[40]) = \mathbb{Q}$, hence $G(120) = G(3) \times G(40)$.
\end{itemize}
From all of this we conclude that
$$
G(120) = \{(g_8,g_3,g_5) \in G(8)\times G(3)\times G(5) : g_8(\zeta_5) = \zeta_5^{\det g_5} \}
$$
hence $E$ has abelian entanglements and $G(120)$ fits into the exact sequence
$$
1 \longrightarrow G(120) \longrightarrow G(8)\times G(3) \times G(5) \longrightarrow \Phi_{120} \longrightarrow 1,
$$
where $\Phi_{120} \simeq (\mathbb{Z}/5\mathbb{Z})^\times$.
Also, given coprime integers $a$ and $f = \prod_\ell \ell^{e_\ell}$ we again set 
$$
m:= \prod_{\ell \mid (f,120)} \ell^{e_\ell} \prod_{\substack{\ell \mid 120 \\ \ell \nmid f}} \ell.
$$

\begin{lem}\label{EtwoZero}
For any $\tilde{\chi} \in \widehat{\Phi}_m - \{1\}$ we have $E_{\chi,2} = 0$.
\end{lem}

\begin{proof}
Suppose first that $4 \nmid f$. Then $m$ is square-free, and because
$$
G(30) = G(2) \times G(3) \times G(5) 
$$
it follows that $\Phi_m \simeq \{ 1 \}$, hence $E_{\chi,2} = 0$. Suppose now that $4 \mid f$, and let $\tilde{\eta}$ be a generator of $\tilde{\Phi}_{120}$. If $8 \mid f$, then $120 \mid m$, hence $\Phi_m \simeq \Phi_{120} \simeq (\mathbb{Z}/5\mathbb{Z})^\times$. Any $\tilde{\chi} \in \widehat{\Phi}_m - \{1\}$ is equal to $\tilde{\eta}^j$ for some $j\in \{1,2,3\}$ and $\chi_2$ is equal to $\eta_2^j$, where
$$
\eta_2 : G(2^{e_2}) \longrightarrow (\mathbb{Z}/5\mathbb{Z})^\times
$$
is the character corresponding to the subfield $\mathbb{Q}(\zeta_5) \subset \mathbb{Q}(E[2^{e_2}])$. Now because $\mathbb{Q}(E[2]) = \mathbb{Q}(i)\subset \mathbb{Q}(\zeta_{2^{e_2}})$ it follows that $\mathbb{Q}(E[2],\zeta_{2^{e_2}}) \cap \mathbb{Q}(\zeta_5) = \mathbb{Q}$, hence
\begin{align*}
\sum_{g \in S(2)} \eta_2^j (g) &= \sum_{x \in (\mathbb{Z}/5\mathbb{Z})^\times} x \\
											    &= 0.
\end{align*}
We conclude that $E_{\chi,2} = 0$. If $4 || f $, then  $\Phi_m \simeq \{\pm 1\}$ and we can use the same argument given that $\mathbb{Q}(i) \cap \mathbb{Q}(\zeta_5) = \mathbb{Q}$. This proves the claim.
\end{proof}

\begin{pro}
For any coprime $(a,f)$ we have that $\mathfrak{C}_E(a,f) = 1$. Further, 
$$
C_E(a,f) = 0 \Longleftrightarrow 4\mid f \text{ and } a \equiv 1 \pmod{4}.
$$
\end{pro}

\begin{proof}
That $\mathfrak{C}_E(a,f) = 1$ follows directly from Theorem \ref{ThmCorrectionFactorProg} and Lemma \ref{EtwoZero}. It follows from this that
$$
C_E(a,f) = \prod_\ell \delta_\ell.
$$ 
For $\ell \neq 2$ we have that $\delta_\ell \neq 0$. Indeed, 
$$
\delta_3 = 
\begin{cases}
\frac{1}{\phi(3^{e_3})} & \text{ if $a \equiv 2 \pmod{3}$ and $3\mid f$} \\
\frac{1}{\phi(3^{e_3})}\big(1-\frac{1}{3}\big) & \text{ if $a \equiv 1 \pmod{3}$ and $3\mid f$} \\
\frac{5}{6} & \text{ if $3\nmid f$}
\end{cases},
$$
and 
$$
\delta_\ell = 
\begin{cases}
\frac{1}{\phi(\ell^{e_\ell})} & \text{ if $a\not\equiv 1 \pmod{\ell}$ and $\ell \mid f$ } \\
\frac{1}{\phi(\ell^{e_\ell})}\Big(1-\frac{1}{\ell(\ell-1)(\ell+1)}\Big) & \text{ if $a \equiv 1 \pmod{\ell}$ and $\ell \mid f$ } \\
1-\frac{1}{(\ell^2-1)(\ell^2-\ell)} & \text{ if $\ell \nmid f$. }
\end{cases}
$$
Finally, given that $\mathbb{Q}(E[2]) = \mathbb{Q}(i)$, it follows that $\delta_2 = 0$ if and only if $4 \mid f$ and $a\equiv 1 \pmod{4}$, and the conclusion follows.
\end{proof}

\begin{rem}
 Suppose $a$ and $f$ are coprime integers such that $a \equiv 1 \pmod{4}$. The above proposition is saying that the only obstruction to the existence of infinitely many primes $p$ such that $\tilde{E}(\mathbb{F}_p)$ is cyclic and $p \equiv a \pmod{f}$ is a local one at the prime $2$. Meaning, for any prime $p$ it is impossible for it to satisfy the required condition at the prime $2$, that is, for $\Frob_p$ to lie in the set $S(2)$, which is the empty set. Note also that even when $f$ is divisible by $4$, we still have $E_{\chi,2} = 0$ and hence $\mathfrak{C}_E(a,f) = 1$. What this is encoding is the fact that $\mathbb{Q}(\zeta_{2^{e_2}}) \cap \mathbb{Q}(\zeta_5) = \mathbb{Q}$ for any $e_2$. The only entanglement of $E$ occurs in the subfield $\mathbb{Q}(\zeta_5)$, and this field is disjoint from $\mathbb{Q}(\zeta_{2^\infty})$.

\end{rem}

\subsection{Example: $Y^2 +XY + Y = X^3-X^2-91X-310$}

So far we have only considered examples where the constant $C_E(a,f)$ either does not vanish, or vanishes because there is a condition at some prime $\ell$ which cannot be satisfied. Another interesting possibility is when all $\delta_\ell$ are non-zero, yet the constant $C_E(a,f)$ still vanishes. This occurs if and only if the entanglement correction factor $\mathfrak{C}_E(a,f)$ vanishes and its expression as a product of local correction factors makes it easy to determine when this happens. The entanglement correction factor being zero means there is an obstruction coming from the entanglement fields which prevent there being infinitely many primes $p$ satisfying the imposed conditions. We will now analyse an example when this occurs.

Consider the elliptic curve  $E$ over $\mathbb{Q}$ given by Weierstrass equation $Y^2 +XY + Y = X^3-X^2-91X-310$. The discriminant of our Weierstrass model is $\Delta = 17$. This curve has one rational torsion point of order $2$ and $\mathbb{Q}(E[2]) = \mathbb{Q}(\sqrt{17})$. In fact, machine computation shows that $m = 34$, where $m$ is the square-free part of $m_E$, and
$$
G(34) = \{(g_2,g_{17}) \in G(2)\times \GL_2(\mathbb{Z}/17\mathbb{Z}) : \varepsilon(g_2) = \theta_{17}\circ \det(g_{17}) \}
$$
where as usual $\varepsilon$ denotes the signature map and $\theta_{17}: (\mathbb{Z}/17\mathbb{Z})^* \rightarrow \{\pm1\}$ denotes the unique quadratic character of $(\mathbb{Z}/17\mathbb{Z})^*$.

If we let $D$ denote the discriminant of $\mathbb{Q}(\sqrt{\Delta})$, then $D=17 \equiv 1 \pmod{4}$, hence by a similar argument to Lemma \ref{deltaEll} we obtain that
$$
\prod_\ell \delta_\ell = \frac{1}{2}\frac{1}{\phi(f)} \prod_{\substack{\ell \mid (a-1,f) \\ \ell \neq 2}}\bigg(1- \frac{1}{\ell(\ell-1)(\ell+1)} \bigg) \prod_{\substack{\ell \nmid f \\ \ell \neq 2}} \bigg(1- \frac{1}{(\ell^2-\ell)(\ell^2-1)} \bigg)
$$
which is non-zero for all $a$ and $f$. By Theorem \ref{ThmCorrectionFactorProg} we have that
$$
C_E(a,f) = \mathfrak{C}_E(a,f) \frac{1}{2}\frac{1}{\phi(f)} \prod_{\substack{\ell \mid (a-1,f) \\ \ell \neq 2}}\bigg(1- \frac{1}{\ell(\ell-1)(\ell+1)} \bigg) \prod_{\substack{\ell \nmid f \\ \ell \neq 2}} \bigg(1- \frac{1}{(\ell^2-\ell)(\ell^2-1)} \bigg)
$$
with
$$
\mathfrak{E}_E(a,f) = 1 + \prod_{\ell \mid 34} E_\ell.
$$
We conclude then the following.
\begin{pro}

For the above elliptic curve we have that $C_E(a,f) = 0$ if and only if $17 \mid f$ and $a$ is a quadratic residue modulo $17$.
\end{pro}

\begin{proof}
The naive density $\prod_\ell \delta_\ell$ is non-vanishing, hence $C_E(a,f) = 0$ if and only if $\mathfrak{C}_E(a,f) = 0$. Using the same argument as in Lemma \ref{OneModFour}, we deduce $E_2 = -1$ for all $a,f$. We have then that
$$
\mathfrak{C}_E(a,f) = 0 \Longleftrightarrow E_{17} = 1.
$$
If $17\nmid f$ then $E_{17} = -1/78335$. If $17\mid f $ then $E_{17} = \blegendre{a}{17}$ and the conclusion follows.
\end{proof}

\begin{rem}
Note that if $17 \mid f$ and $a$ is a quadratic residue mod $17$, then for any prime $p \equiv a \pmod f$ we have that $p$ splits in $\mathbb{Q}(\sqrt{17}) = \mathbb{Q}(E[2])$, so $\Frob_p$ would not satisfy the condition at the prime $2$. The obstruction to the existence of infinitely many primes $p$ such that $\tilde{E}(\mathbb{F}_p)$ is cyclic and $p \equiv a \pmod{f}$ is precisely the entanglement between the $2$ and $17$ torsion fields. The above proposition is saying that this the only obstruction that exists.
\end{rem}

\section{Koblitz's conjecture}\label{KobConj}
In \cite{kob}, N. Koblitz made a conjecture on the asymptotic behaviour of the number of primes $p$ for which the cardinality of the group $\tilde{E}(\mathbb{F}_p)$ is prime. In this section we use our character sum method to give a description of the constants appearing in this asymptotic. 

\begin{conj}[Koblitz]
Let $E/\mathbb{Q}$ be a non-CM curve and let $\Delta$ be the discriminant of any Weierstrass model of $E$ over $\mathbb{Q}$. Suppose that $E$ is not $\mathbb{Q}$-isogenous to a curve with non-trivial $\mathbb{Q}$-torsion. Then 
$$
|\{\textrm{primes } p\leqslant x : p\nmid \Delta, |\tilde{E}(\mathbb{F}_p)|\textrm{ is prime} \} | \sim C_E \frac{x}{(\log x)^2}
$$
as $x\rightarrow \infty$ where $C_E$ is an explicit positive constant. 
\end{conj}

In \cite{ZywKob}, Zywina shows that the description of the constant $C_E$ given by Koblitz is not always correct, and he gives a corrected description of the constant along with providing several interesting examples and numerical evidence for the refined conjecture. In particular the constant described by Zywina is not necessarily positive. The reason the original constant is not always correct is that it does not take into account that divisibility conditions modulo distinct primes need not be independent. Put another way, it could occur that there are non-trivial intersections between distinct $\ell$-power torsion fields of $E$. The following is the refined Koblitz conjecture given by Zywina, which here we state restricted to non-CM curves over $\mathbb{Q}$. 

\begin{conj}\label{RefinedConj}
Let $E/\mathbb{Q}$ be a non-CM elliptic curve of discriminant $\Delta$, and let $t$ be a positive integer. Then there is an explicit constant $C_{E,t}\geqslant 0$ such that
$$
|\{\textrm{primes } p\leqslant x : p\nmid \Delta, |\tilde{E}(\mathbb{F}_p)|/t\textrm{ is prime} \} | \sim C_{E,t} \frac{x}{(\log x)^2}
$$
as $x\rightarrow \infty$. 
\end{conj}

The condition on $p$ that $|\tilde{E}(\mathbb{F}_p)|/t$ be prime can be given as a splitting condition in the various $\ell$-torsion fields, so the character sum method we have developed again seems well suited to compute $C_{E,t}$.  In his paper Zywina computes the constants $C_{E,t}$ via a different method than the one we use here, both in the CM and non-CM cases. Here we will restrict ourselves to non-CM curves with abelian entanglements over the rationals. 

For each prime power $\ell^\alpha$, define
$$
\Psi_t(\ell^\alpha) := \left \{ A \in \GL_2(\mathbb{Z}/\ell^\alpha\mathbb{Z}) :\ \det{(I-A)} \in t \cdot \big (\mathbb{Z}/\ell^\alpha\mathbb{Z}\big)^\times \right \}.
$$
For a prime $p \nmid N_E \ell$ note that $\tilde{E}(\mathbb{F}_p)/t$ is invertible modulo $\ell^\alpha/(\ell^\alpha,t)$ if and only if $\rho_{\ell^\alpha}(\Frob_p) \in G(\ell^\alpha) \cap \Psi_t(\ell^\alpha)$. Suppose that $t$ has prime factorisation $t = \prod_\ell \ell^{e_\ell}$. With this in mind, define the set of `good' Frobenius elements to be
$$
S_t(\ell) = 
\begin{cases}
G(\ell^{e_\ell+1}) \cap \Psi_t(\ell^{e_\ell+1}) & \text{ if $\ell \mid t$ } \\
G(\ell) \cap \Psi_t(\ell) & \text{ if $\ell \nmid t$ }.
\end{cases}
$$

\begin{comment}
We have then that $\tilde{E}(\mathbb{F}_p)/t$ is prime if and only if for every $\ell \nmid p$ the following holds:
\begin{itemize}
\item[(i)]
$\rho_\ell (\Frob_p) \in S(\ell)$ if $\ell \nmid t$,
\item[(ii)]
$\rho_{\ell^{e_\ell+1}}(\Frob_p) \in S(\ell)$ if $\ell \mid t$.
\end{itemize}
To see this, note first that that $\tilde{E}(\mathbb{F}_p)/t$ is an integer if and only if $\rho_{\ell^{e_\ell}}(\Frob_p) \in G(\ell^{e_\ell}) \cap \Psi_t(\ell^{e_\ell})$ for all $\ell \mid t$. In addition, $\tilde{E}(\mathbb{F}_p)/t$ is a prime number if and only if it is also invertible modulo all primes $\ell < \tilde{E}(\mathbb{F}_p)/t$.
\end{comment}

We now give a description of the constant $C_{E,t}$ in terms of our sets $S_t(\ell)$ as well as a crude heuristic of justifying it. This heuristic follows the same lines as that of Koblitz and Zywina. The key argument relies on the Cramer's model which asserts that, roughly speaking, the primes behave as if every random integer $n$ is prime with probability $1/\log n$. If the sequence $\{ |\tilde{E}(\mathbb{F}_p)|/t\}_{p \nmid N_E}$ were assumed to behave like random integers, then the proability that $|\tilde{E}(\mathbb{F}_p)|/t$ is prime would be
$$
\frac{1}{\log\big(|\tilde{E}(\mathbb{F}_p)|/t\big)} \approx \frac{1}{\log(p+1) - \log t}.
$$
The last approximation uses the fact that by Hasse's bound, $\tilde{E}(\mathbb{F}_p)$ is close to $p+1$. 

It is not true however, that the $|\tilde{E}(\mathbb{F}_p)|/t$ behave like random integers with respect to congruences, and in order to get a better approximation we need to take these congruences into account. If we fix a prime $\ell$, then for all but finitely many $p$. if $|\tilde{E}(\mathbb{F}_p)|/t$ is prime then it is invertible modulo $\ell$. If $\ell$ does not divide $t$, then by Chebotarev, the density of primes $p\nmid N_E$ such that $\tilde{E}(\mathbb{F}_p)/t$ is invertible modulo $\ell$ is $|S_t(\ell)|/|G(\ell)|$. If $\ell \mid t$, then similarly the density of primes $p\nmid N_E$ such that $\tilde{E}(\mathbb{F}_p)$ is divisible by $\ell^{e_\ell}$ \emph{and} $\tilde{E}(\mathbb{F}_p)/t$ is invertible modulo $\ell$ equals $|S_t(\ell)|/|G(\ell^{e_\ell+1})|$. Meanwhile the density of natural numbers that are invertible mod $\ell$ is $(1-1/\ell)$. If we let $d$ be a square-free integer coprime to $t$, then 
$$
\prod_{\ell \mid td} \frac{1}{1-1/\ell}\prod_{\ell\mid t} \frac{|S_t(\ell)|}{|G(\ell^{e_\ell+1})|}  \prod_{\ell \mid d} \frac{|S_t(\ell)|}{|G(\ell)|} \cdot \frac{1}{\log(p+1) - \log t}
$$
should constitute a better approximation to the probability that $|\tilde{E}(\mathbb{F}_p)|/t$ is prime, as it takes into account the congruences modulo all primes $\ell \mid td$. Taking into account all congruences amounts to letting $d$ tend to infinity, hence this model suggests that for a randomly chosen $p$, $|\tilde{E}(\mathbb{F}_p)|/t$ is prime with probability
$$
\prod_\ell \frac{\delta_\ell}{1-1/\ell} \cdot \frac{1}{\log(p+1) - \log t}
$$
where
$$
\delta_\ell = 
\begin{cases}
|S_t(\ell)|/|G(\ell)| & \text{ if $\ell \nmid t$ } \\
|S_t(\ell)|/|G(\ell^{e_\ell+1})| & \text{ if $\ell \mid t$ }. \\
\end{cases}
$$

This is the constant that was given by Koblitz with $t=1$ and later refined by Zywina. The problem that still remained with the approximation given by Koblitz, is that while it does take into account congruences modulo $\ell$, is assumes that divisibility conditions modulo distinct primes are independent. In order to deal with this we take a similar approach as in the previous sections. That is, we let 
$$
m:= \prod_{\ell \mid t} \ell^{e_\ell+1} \prod_{\substack{\ell \mid m_E \\ \ell \nmid t}} \ell
$$
and for each square-free $d$ coprime to $m$, let
$$
\mathcal{S}_{md} := \prod_{\ell \mid md} S_t(\ell), \quad \mathcal{G}_{md} := \prod_{\ell \mid t} G(\ell^{e_\ell+1}) \prod_{\substack{\ell \mid md \\ \ell \nmid t}} G(\ell).
$$
By Corollary \ref{CorEquivEntang}
$$
G(md) \leqslant \mathcal{G}_{md}
$$
has abelian entanglements, hence we have an exact sequence
$$
1 \longrightarrow G(md) \longrightarrow \mathcal{G}_{md} \xrightarrow{\phantom{a}\psi_{md}\phantom{b}} \Phi_{md} \longrightarrow 1
$$
for some abelian group $\Phi_{md}$. By (\ref{ConstantChi}) we have that $\Phi_{md} \simeq \Phi_m$ for any square-free $d$ coprime to $m$. Note now that $|\mathcal{S}_{md}\cap G(md)|/|G(md)|$ is the density of $p$ for which $|\tilde{E}(\mathbb{F}_p)|/t$ is an integer and invertible modulo $md$, hence by letting $d$ tend to infinity over the square free integers coprime to $m$, the refined constant is
\begin{align*}
C_{E,t} &= \lim_{d\rightarrow \infty} \frac{|\mathcal{S}_{md}\cap G(md)|/|G(md)|}{1-1/\ell} \\
			 &= \left( \prod_{\ell\mid m} \frac{1}{1-1/\ell} \right ) \cdot \frac{|\mathcal{S}_m\cap G(m)|}{|G(m)|} \prod_{\ell \nmid m} \frac{\delta_\ell}{1-1/\ell}.
\end{align*}
It follows by the prime number theorem that the expected number of primes $p$ such that $|\tilde{E}(\mathbb{F}_p)|/t$ is prime is asymptotic to $C_{E,t}\cdot  x/(\log x)^2$.

Applying Theorem \ref{ThmGoodFrobeniusCalc} with $m$ defined as above we obtain
\begin{equation}\label{KoblitzEntanglements}
C_{E,t} = \mathfrak{C}_{E,t} \prod_\ell \frac{\delta_\ell}{1-1/\ell} 
\end{equation}
where the entanglement correction factor $\mathfrak{C}_{E,t}$ is given by
$$
\mathfrak{C}_{E,t} = 1+\sum_{\tilde{\chi}\in \widehat{\Phi}_m-\{1\}}\prod_{\ell | m} E_{\chi,\ell}.
$$

\subsection{Serre curves}
In this section we compute the constants $C_{E,1}$ in Conjecture \ref{RefinedConj} for Serre curves. This will amount to finding the average value of various quadratic characters on $S(\ell)$. In the case of Serre curves, the sets $S(\ell)$ are particularly easy to treat.

\begin{pro}
Let $E/\mathbb{Q}$ be a Serre curve. Let $D$ be the discriminant of $\mathbb{Q}(\sqrt{\Delta})$ where $\Delta$ is the discriminant of any Weierstrass model of $E$ over $\mathbb{Q}$. Then
$$
C_{E,1} = \mathfrak{C}_{E,1} \prod_{\ell}\bigg(1-\frac{\ell^2-\ell-1}{(\ell-1)^3(\ell+1)} \bigg )
$$
where the entanglement correction factor $\mathfrak{C}_{E,1}$ is given by
$$
\mathfrak{C}_{E,1} =
\begin{cases}
1 & \text{ if $D\equiv 0 \pmod{4}$ } \\
1+ \displaystyle{\prod_{\ell \mid D} \frac{1}{\ell^3-2\ell^2-\ell+3}} & \text{ if $D\equiv 1 \pmod{4}$ }
\end{cases}
$$
\end{pro}

\begin{proof}
We begin by noting that, for Serre curves, 
$$
S_1(\ell) = \left \{ A \in \GL_2(\mathbb{Z}/\ell\mathbb{Z}) :\ \det{(I-A)} \in \big (\mathbb{Z}/\ell\mathbb{Z}\big)^\times \right \}.
$$
We have then that
\begin{align*}
\delta_\ell &= \frac{|S_1(\ell)|}{|G(\ell)|}  \\
				 &= 1-\frac{|S_1(\ell)^c|}{|\GL_2(\mathbb{Z}/\ell\mathbb{Z})|}
\end{align*}
where $S_1(\ell)^c = \left \{ A \in \GL_2(\mathbb{Z}/\ell\mathbb{Z}) :\ \det{(I-A)}=0 \right \}$. Thus $S_1(\ell)^c$ consists of those matrices whose eigenvalues are $1$ and $\lambda$ for some $\lambda \in (\mathbb{Z}/\ell\mathbb{Z})^\times$. It follows from Table 12.4 in \S 12, Chapter XVIII of \cite{LangAlgebra}, that there are $\ell^2$ elements of $\GL_2(\mathbb{Z}/\ell\mathbb{Z})$ with both eigenvalues equal to $1$, and $\ell^2+\ell$ elements with eigenvalues $1$ and $\lambda\neq 1$. We obtain then that $|S_1(\ell)^c| = \ell^2 + (\ell-2)(\ell^2+\ell)$, hence we have that
$$
\delta_\ell = 1- \frac{\ell^2 + (\ell-2)(\ell^2+\ell)}{(\ell^2-\ell)(\ell^2-1)}
$$
and a calculation yields that
$$
\frac{\delta_\ell}{1-1/\ell} = 1-\frac{\ell^2-\ell-1}{(\ell-1)^3(\ell+1)}.
$$
From (\ref{KoblitzEntanglements}) it rests only to compute $\mathfrak{C}_{E,1}$. Because $t=1$, $m$ equals the square-free part of $m_E$, and we may proceed just as in the proof of Proposition \ref{SerreCurveFirstCase}. That is, when $D\equiv 0 \pmod{4}$ then $\mathfrak{C}_{E,1} = 1$. If $ D\equiv 1 \pmod{4}$,  then for  each $\ell \mid 2D$ it suffices to compute the average value of $\chi_\ell$ on $S_1(\ell)$. 

Note that since the $\chi_\ell$ are non-trivial, then $\sum_{x\in G(\ell)} \chi_\ell (x) = 0$. For $\ell>2$ recall that $\chi_\ell = \blegendre{\det}{\ell}$, hence given an element $x\in S_1(\ell)^c$ with eigenvalues $1$ and $\lambda$, we have that $\chi_\ell(x) = \blegendre{\lambda}{\ell}$. There are an equal number of squares and non-squares in $(\mathbb{Z}/\ell\mathbb{Z})^\times$, so we conclude then 
\begin{align*}
\sum_{x\in S_1(\ell)} \chi_\ell(x) &= - \sum_{x\in S_1(\ell)^c} \chi_\ell(x) \\
													&= - \bigg( \ell^2 \blegendre{1}{\ell} + (\ell^2+\ell)\sum_{\substack{\lambda \in (\mathbb{Z}/\ell\mathbb{Z})^\times \\ \ell \neq 1}} \blegendre{\lambda}{\ell} \bigg ) \\
													&= - \big( \ell^2 - (\ell^2+\ell) \big) \\
													&= \ell.
\end{align*}
From this we obtain 
\begin{align*}
E_\ell &= \frac{\ell}{|G(\ell)| - |S_1(\ell)|} \\
		 &= \frac{\ell}{(\ell^2-\ell)(\ell^2-1) - (\ell^2+\ell)(\ell-2) - \ell^2}\\
		 &= \frac{1}{\ell^3-2\ell^2-\ell+3}.
\end{align*}
For $\ell = 2$ one can directly compute $S_1(2)$. It consists of the $2$ matrices \begin{tiny} $\begin{pmatrix} 1 & 1 \\ 1 & 0 \end{pmatrix}$\end{tiny} and \begin{tiny} $\begin{pmatrix} 0 & 1\\ 1 & 1 \end{pmatrix}$ \end{tiny}  both of which have order $3$ and hence are even permutations. Since $\chi_2$ is the signature character we conclude $E_2 = 1$, and this completes the proof.
\end{proof}

\section{Elliptic curves without abelian entanglements}\label{NonAbEntang}

As we have seen, the character sum method we have developed for the study of conjectural constants can only be applied to the class of elliptic curves with abelian entanglements. However this condition does not seem very restrictive, given that at least `most' elliptic curves over $\mathbb{Q}$ satisfy this property, and in fact it is not even clear whether or not there exists an elliptic curve which does not. In this final section we show that there does indeed exist at least one infinite family of curves which do not satisfy the abelian entanglements property. The character sum method as we have developed it cannot be applied to the curves in this family, however we will see that with some additional restrictions it still can be. 

The infinite family we exhibit is studied in \cite{BrauTorsion}. We sketch the construction here. For more details please refer to the aforementioned paper. It is parametrised by a modular curve $X_H$ of level $6$, which we now describe. Let $\mathcal{N}\subset\GL_2(\mathbb{Z}/3\mathbb{Z})$ be the subgroup generated by the set
$$
\left \{ \begin{pmatrix} 0 & 2 \\ 1 & 0 \end{pmatrix},\ \begin{pmatrix} 1 & 2 \\ 2 & 2 \end{pmatrix},\ \begin{pmatrix} 2 & 0 \\ 0 & 2 \end{pmatrix} \right \}.
$$
Then $\mathcal{N}$ is the unique index $6$ normal subgroup of $\GL_2(\mathbb{Z}/3\mathbb{Z})$, which fits into the exact sequence
$$
1 \longrightarrow \mathcal{N} \longrightarrow \GL_2(\mathbb{Z}/3\mathbb{Z}) \xrightarrow{\phantom{a}\theta\phantom{b}} \GL_2(\mathbb{Z}/2\mathbb{Z}) \longrightarrow 1
$$
for some non-canonical map $\theta:\GL(\mathbb{Z}/3\mathbb{Z}) \rightarrow \GL_2(\mathbb{Z}/2\mathbb{Z})$. Let 
$$
H:= \{ (g_2,g_3) \in \GL_2(\mathbb{Z}/2\mathbb{Z}) \times \GL_2(\mathbb{Z}/3\mathbb{Z}) : g_2 = \theta(g_3) \}  \subset \GL_2(\mathbb{Z}/6\mathbb{Z})
$$
be the graph of $\theta$, viewed as a subgroup of $\GL_2(\mathbb{Z}/6\mathbb{Z})$ via the Chinese Remainder Theorem.

It follows that for every elliptic curve $E$ over $\mathbb{Q}$,
$$
j(E) \in j(X_H(\mathbb{Q})) \; \Longleftrightarrow \; E \simeq_{\overline{\mathbb{Q}}} E' \; \text{ for some } E' \text{ over } \mathbb{Q} \text{ for which } \mathbb{Q}(E'[2]) \subset \mathbb{Q}(E'[3]).
$$
The curve $X_H$ is seen to have genus $0$ and one cusp, which must then be defined over $\mathbb{Q}$, thus we have $X_H\simeq \mathbb{P}^1$. The following theorem gives an explicit model for $X_H$. We omit the proof, which can be found in \cite{BrauTorsion}.

\begin{thm}
There is a uniformiser at the cusp
$
t : X_H \longrightarrow \mathbb{P}^1
$
with the property that
$$
j = 2^{10} 3^3 t^3 (1 - 4t^3),
$$
where $j : X_H \longrightarrow X(1) \simeq \mathbb{P}^1$ is the usual $j$-map.  
\end{thm}

Let us take one example from this family. Consider the curve $E/\mathbb{Q}$ given by minimal Weierstrass equation $Y^2 = X^3 - 63504X + 6223392$. This curve has $j(E) = -2^{10} 3^4$, as well as $\Delta = -2^4 3^{11} 7^6$. Machine computation shows that $G(\ell) = \GL_2(\mathbb{Z}/\ell\mathbb{Z})$ and $\mathbb{Q}(E[2]) \subset \mathbb{Q}(E[3])$. We also have that $\mathbb{Q}(\sqrt{\Delta}) = \mathbb{Q}(\sqrt{-3})$, which is what we expect since the maximal abelian extension inside $\mathbb{Q}(E[3])$ is precisely $\mathbb{Q}(\sqrt{-3})$. 

Suppose we wish to compute the conjectural density of primes $p$ such that $\tilde{E}(\mathbb{F}_p)$ is cyclic. As we have seen, the naive density of this is $\prod_\ell \delta_\ell$, however a correction factor is needed. As the only critical primes are $2,3$ and $7$, the density we are looking for is 
$$
C_E = \frac{|G(42)\cap \mathcal{S}_{42}|}{G(42)|} \prod_{\ell \neq 2,3,7} \delta_\ell,
$$
where we are using the notation of Section \ref{CyclicRedEC}. Now $\GL_2(\mathbb{Z}/3\mathbb{Z})$ and $\GL_2(\mathbb{Z}/7 \mathbb{Z})$ have no common simple non-abelian quotients, hence any entanglement between the fields $\mathbb{Q}(E[3])$ and $\mathbb{Q}(E[7])$ would have to contain a non-trivial abelian subfield. However the maximal abelian extensions of $\mathbb{Q}(E[3])$ and $\mathbb{Q}(E[7])$ are $\mathbb{Q}(\zeta_3)$ and $\mathbb{Q}(\zeta_7)$, hence we conclude $\mathbb{Q}(E[3]) \cap \mathbb{Q}(E[7]) = \mathbb{Q}$. This implies that $G(42) = G(6) \times G(7)$, hence
$$
C_E = \frac{|G(6)\cap \mathcal{S}_{6}|}{|G(6)|} \prod_{\ell \neq 2,3} \delta_\ell,
$$
Finally, note that because $G(6) = G(3)$ and $G(2)$ is a quotient of $G(6)$, then
$$
\frac{|G(6)\cap \mathcal{S}_{6}|}{|G(6)|} = \frac{|S(2)|}{|G(2)|}.
$$
Using machine computation we find that the observed density of primes $p\leqslant 100000000$ is $0.831069$ while our computation yields
\begin{align*}
C_E &= \prod_{\ell \neq 3} \delta_\ell \\
	   &\approx 0.831066.
\end{align*}

As mentioned in the introduction, another natural question which arises from this is whether one can one classify the triples $(E,m_1,m_2)$ with $E$ an elliptic curve over $\mathbb{Q}$ and $m_1,m_2$ a pair of coprime integers for which the entanglement field $\mathbb{Q}(E[m_1]) \cap \mathbb{Q}(E[m_2])$ is non-abelian over $\mathbb{Q}$. We are not sure if any other families exist, however one systematic way one could possibly rule out other examples is via the following steps.
\begin{itemize}
\item[(i)]
Classify the non-abelian groups which arise as common quotients of subgroups $H_{m_1}$ and $H_{m_2}$, where $H_{m_i} \subset \GL_2(\mathbb{Z}/m_i \mathbb{Z})$ and $\det(H_{m_i}) = (\mathbb{Z}/m_i\mathbb{Z})^\times$ for $i=1,2$.
\item[(ii)]
For each example in step (i), compute the genus of the associated modular curve.
\item[(iii)]
For each modular curve in step (ii), decide whether or not it has any rational points. 
\end{itemize}

For each of these families of curves it would also be of interest to find a systematic way to compute their entanglement correction factors. For the family we have described here this is easy to do because one of the torsion fields is fully contained in another one. It may occur however, at least in theory, that a curve could have many non-abelian intersections between various of its torsion fields. However it seems unlikely many examples of this type exist.

\bibliographystyle{amsalpha}
\bibliography{charSumsBib.bib}

\end{document}